\input ./style/arxiv-general.cfg
\documentclass[aop,MSNbibl,dvips]{arximspdf}
\makeatletter
   \@ifpackageloaded{graphicx}{}{\usepackage{graphicx}}
\makeatother

\doi{10.1214/14-AOP968} 
\volume{44}
\issue{1}
\pubyear{2016}
\firstpage{107}
\lastpage{130}
\docsubty{FLA}

\makeatletter
\newcommand{\eqref}[1]{(\ref{#1})}
\newtheorem{theorem}{Theorem}
\newtheorem{lemma}[theorem]{Lemma}
\newtheorem{proposition}[theorem]{Proposition}
\newproclaim{conjecture}{Conjecture}
\newproclaim{question}[conjecture]{Question}
\newproclaim{remark}[theorem]{Remark}
\makeatother

\begin{document}
\begin{frontmatter}

\title{A monotone Sinai theorem\thanksref{T1}}
\runtitle{A monotone Sinai theorem}

\begin{aug}
\author[A]{\fnms{Anthony}~\snm{Quas}\ead[label=e1]{aquas@uvic.ca}\ead[label=u1,url]{http://www.math.uvic.ca/\textasciitilde aquas/}}
\and
\author[B]{\fnms{Terry}~\snm{Soo}\corref{}\ead[label=e2]{tsoo@ku.edu}\ead[label=u2,url]{http://www.math.ku.edu/u/tsoo}}
\runauthor{A. Quas and T. Soo}
\affiliation{University of Victoria and University of Kansas}
\address[A]{Department of Mathematics and Statistics \\
University of Victoria \\
P.O. BOX 3060 STN CSC \\
Victoria, British Columbia V8W 3R4 \\
Canada\\
\printead{e1} \\
\printead{u1}} 
\address[B]{Department of Mathematics \\
University of Kansas \\
405 Snow Hall \\
1460 Jayhawk Blvd \\
Lawrence, Kansas 66045-7594\\
USA \\
\printead{e2}\\
\printead{u2}}
\end{aug}
\thankstext{T1}{Funded in part by NSERC and MSRI.}

\received{\smonth{11} \syear{2013}}
\revised{\smonth{9} \syear{2014}}

%
\begin{abstract}
Sinai proved that a nonatomic ergodic measure-preserving system has any
Bernoulli shift of no greater entropy as a factor. Given a Bernoulli
shift, we show that any other Bernoulli shift that is of strictly
less entropy and is stochastically dominated by the original measure
can be obtained as a \textit{monotone} factor; that is, the factor map
has the
property
that for each point in the domain, its image under the factor map is
coordinatewise
smaller than or equal to the original point.
\end{abstract}

%
\end{frontmatter}
%


\section{Introduction}

Let $(\mathrm{X}, \mu)$ be a probability space. If $T\dvtx \mathrm{X} \to
\mathrm{X}$ is a map such
that $\mu\circ T^{-1}= \mu$, then $(\mathrm{X}, \mu, T)$ is a
measure-preserving
system, and if every almost-surely $T$-invariant set has measure zero
or one,
then the system is ergodic. Let $S$ be a self-map of a measurable space
$Y$. A measurable mapping $\phi\dvtx  \mathrm{X} \to\mathrm{Y}$ such
that $\mu\circ\phi^{-1} = \nu$ and $\phi\circ T = S \circ\phi$ on a
subset of $\mu$-full measure is a
\textit{factor map};
when a factor map exists, we say that $(\mathrm{Y}, \nu, S)$ is a
\textit{factor}
of $( \mathrm{X}, \mu, T)$. It is well known that in this case, $h(\nu
)\le h(\mu)$,
where $h$ is the (Kolmogorov--Sinai) entropy.
For a positive integer $N$, let
$[N]:=\{0,1, \ldots, N-1\}$. If $\mathrm{Y}=[N]^{\mathbb{Z}}$ is the
space of all bi-infinite sequences of a finite number of symbols and
$\nu= p^{\mathbb{Z}}$ for some nontrivial probability measure $p =
(p_i)_{i=0} ^{N-1}$ on
$[N]$, and $S$ is the left-shift given by $S(\mathrm{y})_i = \mathrm{y}_{i+1}$ for all
$i \in\mathbb{Z}$, then we say that $B(p) := (\mathrm{Y}, \nu, S)$ is
a~\textit{Bernoulli shift}
on $N$ symbols and that $\nu$ is a
\textit{Bernoulli measure}.
The entropy of the Bernoulli shift $B(p)$
is given by the positive number
\[
H(p):=-\sum_{i=0}^{n-1} p_i
\log p_i.
\]

Sinai \cite{Sinai,MR2766434} proved that if $(\mathrm{X}, \mu, T)$ is a
nonatomic
invertible
ergodic measure-preserving system of entropy $h>0$, then it has any Bernoulli
shift of any entropy $h'\le h$ as a factor.

Let $(E,\succeq)$ be a partially ordered Polish
space such that the set $M:=\{(x, x') \in E^2\dvtx  x \succeq x'\}$
is closed in the product topology. For two probability measures
$\alpha$ and $\beta$ on $E$,\vadjust{\goodbreak} we say that $\alpha$ \textit{stochastically dominates} $\beta$
if the integrals with respect to $\alpha$ and $\beta$ satisfy $\alpha
(f) \geq\beta(f)$ for all nondecreasing bounded functions $f \dvtx  E \to
\mathbb{R}$. By Strassen's theorem \cite{Strassen}, this is equivalent
to the existence of a~monotone coupling of $\alpha$ and $\beta$; that
is, a
measure $\rho$ on $E\times E$ whose projection on the first factor is
$\alpha$,
on the second factor is $\beta$,
and satisfies $\rho(M)=1$.

In our context, the partial order is defined by $\mathrm{x}\succeq
\mathrm{x}'$ if $\mathrm{x}_i\ge\mathrm{x}'_i$ for each
$i\in\mathbb{Z}$. It is well known that $p^{\mathbb{Z}}$
stochastically dominates $q^{\mathbb{Z}}$
if and only if $p$ stochastically dominates $q$ (where the partial
order on\vspace*{1pt}
$[N]$ is the standard total order), and $p$~stochastically dominates
$q$ if and only if
$\sum_{i=0}^k p_i \leq \sum_{i=0}^k q_i$ for all
$0 \leq k < N$.
A~factor map, mapping $[N]^{\mathbb{Z}}$ to itself, is\vspace*{1pt} said
to be \textit{monotone} if $\phi(\mathrm{x})\preceq\mathrm{x}$ for
each $\mathrm{x}\in[N]^{\mathbb{Z}}$. In our context,
by the definition of a factor, this is equivalent to the condition
$\phi(\mathrm{x})_0\le\mathrm{x}_0$ for all $\mathrm{x}\in[N]^{\mathbb{Z}}$.
Notice that if $\phi$ is a monotone factor from $([N]^{\mathbb{Z}},\mu)$
to $([N]^{\mathbb{Z}},\nu)$, then $\mu$ stochastically dominates $\nu$.

It follows from the above that two necessary conditions for
$B(q)$ to be a monotone factor of $B(p)$ are that $p$ stochastically
dominates $q$ and $H(p)\ge H(q)$.
Karen Ball and Russell Lyons \cite{Ball} asked about a partial converse:

\begin{quote}
If $p$ and $q$ are probability measures on $[N]$ such that $p$
stochastically dominates $q$
and $H(p)>H(q)$,
does there exist a monotone factor map from $B(p)$ to $B(q)$?
\end{quote}

We answer this question affirmatively.

%
\begin{theorem}
\label{result}
Let $B(p)$ and $B(q)$ be Bernoulli shifts with symbols in $[N]$
(where one allows the possibility that
$p$ and $q$ give zero mass to some symbols). If the entropy of
$B(p)$
is strictly greater than that of $B(q)$ and the measure $p$
stochastically dominates $q$, then $B(q)$ is a monotone factor of $B(p)$.
\end{theorem}

Ball \cite{Ball} proved Theorem~\ref{result} in the special case
where $q$ only assigns positive mass to the two symbols $\{0,1\}$,
and also in the case where the entropy of $B(p)$ is greater than
logarithm of the total number of symbols with positive
$q$-mass; in particular, this implies that if $n>k$ and $q_i = 1/k$ for
all $ 0 \leq i < k$ and $p_i = 1/n$ for all $0 \leq i < n$,
then $B(q)$ is a monotone factor of $B(p)$. Ball's proof worked by adapting
and extending the methods of the Keane and Smorodinsky \cite{keanea,keaneb}
proof of the Sinai theorem for case of Bernoulli shift.
We will make use of a monotone coupling of two Bernoulli shifts that was
defined by Ball (see Section~\ref{ball}), having a useful product structure
and independence properties.

Another idea that we make use of comes from
del Junco's proof \cite{Juncoa,Junco} of the Sinai
theorem. He replaces the combinatorial
marriage theorem (see Section~\ref{hall}) used by Keane and Smorodinsky
and by Ball
with a ingenious variation of the quantile coupling (see Section~\ref{delstar})
that we adapt to handle monotonicity.
Our proof will also make use of a version of the marriage theorem of
Keane and Smorodinsky, but in a more limited way.
By combining the tools of Ball and del Junco, we are able to use the
Burton--Rothstein \cite{BurKeaSer,BurRot} method to produce a monotone factor
using the Baire category theorem.

Before we discuss in more detail the idea
of the proof of Theorem~\ref{result} and other related results in ergodic
theory and probability in the next sections, we ask a few questions and
state an extension of Theorem~\ref{result}, where stochastic domination
is replaced by a general relation.

%
\begin{question}
Is\vspace*{1pt} Theorem~\ref{result} true if we allow for the possibility
that the entropy of $B(p)$ is equal to the entropy of $B(q)$?
For example, if $p_0=\frac{1}3, p_1=\frac{2}3$, $q_0 = \frac{2}3$ and
$q_1=\frac{1}3$,
we do not know whether $B(q)$ is a monotone factor of $B(p)$.
Is it possible that there is a monotone factor that is also an isomorphism?
\end{question}

Motivated by Theorem~\ref{result} and the fact that Sinai's theorem
does not require the original space to be a Bernoulli shift, we
ask if the following monotone Sinai-type theorems are true.

\begin{question}
\label{monosinai}
Let $B(q)$ be a Bernoulli shift on $[N]$, and let
$\mu$ be an ergodic shift-invariant nonatomic measure on $[N]^{\mathbb{Z}}$
which stochastically dominates the product measure $q^{\mathbb{Z}}$,
and assume
that entropy of the system with the measure $\mu$ is no less than the
entropy of $B(q)$. Sinai's theorem gives that $B(q)$ can always
be obtained as a factor of the system with the measure $\mu$, but can
it be obtained as a~monotone factor?
\end{question}

%
\begin{question}
Suppose $\mu$ and $\nu$ are ergodic shift-invariant
nonatomic measures on $[N]^{\mathbb{Z}}$, where $\mu$ stochastically
dominates $\nu$. Assume that the system with the measure $\nu$ can
be obtained as a factor of the system with measure $\mu$; must there
exist a monotone factor? For example, the stationary bi-infinite Markov
process with transition probabilities given by $q_{00}=\frac{1}2=
q_{01},
q_{10}=\frac{2}3$ and $q_{11}=\frac{1}3$ can be obtained as a factor of
the Bernoulli
shift $B(p)$, where $p_0=p_1=\frac{1}2$ \cite{MR525312}, and it is also
easy to
see that associated Markovian measure $\nu$ is stochastically dominated by
the product measure $\mu=p^{\mathbb{Z}}$.
\end{question}

Let $R \subset[N] \times[N]$ be a relation on $[N]$. Let $p$ and $q$
be probability measures on~$[N]$. Motivated by Strassen's theorem and a
question raised by Gurel-Gurevich and Peled \cite{GGP}, Section~1.3, we
say that $p$ $R$-\textit{dominates} $q$ if there exists a probability
measure $\rho$ on
$[N] \times[N]$ which gives unit mass to set $R$, and has projections
equal to $p$ and $q$ on the first and second copies of $[N]$,
respectively. We call the measure $\rho$ an $R$-\textit{coupling}.

It is an interesting question of Gurel-Gurevich and Peled \cite{GGP}, Section~1.3,  who ask, in the general setting of Borel spaces
$(B_1, \rho_1)$ and $(B_1, \rho_2)$, for what relations $R \subset B_1
\times B_2$, does the existence of a $R$-coupling imply the existence
of a~\textit{deterministic} $R$-coupling; that is, a coupling $\rho$ for
which there exists a function $f \dvtx B_1 \to B_2$ such that
$\rho\{(x, f(x))\dvtx B_1\}=1$. We prove the following related result in
the more restricted context of Bernoulli factors.

%
\begin{theorem}
\label{relation}
Let $B(p)$ and $B(q)$ be Bernoulli shifts with symbols in $[N]$
(where one allows the possibility that
$p$ and $q$ give zero mass to some symbols). Let $R$ be any relation on
$[N]$. If the entropy of
$B(p)$
is strictly greater than that of $B(q)$, and the measure $p$
$R$-dominates the measure $q$, then there exists a factor $\phi$ from
$B(p)$ to $B(q)$ such that $(\mathrm{x}_0, \phi(\mathrm{x}_0)) \in R$
for all $\mathrm{x} \in[N]^{\mathbb{Z}}$.
\end{theorem}

We will see that the proof of Theorem~\ref{relation} does not require
any additional work; we will point out the necessary modifications to
the proof of Theorem~\ref{result}, and concentrate on the case of
stochastic domination.

\section{Background}
%

\subsection{The isomorphism problem for Bernoulli shifts}
\label{iso}
Let $(\mathrm{X}, \mu, T)$ and $(\mathrm{Y},  \nu, S)$ be
measure-preserving systems.
A factor $\phi \dvtx \mathrm{X} \to\mathrm{Y}$ is an \textit{isomorphism} if
$\phi^{-1} \dvtx  \mathrm{Y} \to\mathrm{X}$
is also a factor. A fundamental question in ergodic theory
is to ask when are two systems
isomorphic \cite{MR0122959,MR0304616}.
It was an open question whether the two Bernoulli shifts given by
$p=(\frac{1}2,\frac{1}2)$ and $q=(\frac{1}3,\frac{1}3,\frac{1}3)$ were
isomorphic, until
Kolmogorov gave a negative answer by introducing the idea of entropy
from statistical physics into ergodic theory and proving that
(Kolmogorov--Sinai) entropy is an isomorphism invariant \cite{MR2342699}.
Sinai's theorem and isomorphisms constructed for certain specific
cases by Me{\v{s}}alkin \cite{MR0110782}, \cite{MR832433}, page 181,
and Blum and Hansen \cite{MR0143862}
suggested that entropy could be a complete isomorphism-invariant for
Bernoulli shifts. Ornstein \cite{Ornstein,MR3052869} proved that
this was true; any two Bernoulli shifts of equal entropy are isomorphic.

\subsection{Joinings and Baire category}
It is an easy application of the Baire category theorem to prove
the existence of a continuous and nowhere differentiable function.
Burton and Rothstein
\cite{BurKeaSer,BurRot} had the nice idea to use the Baire
category theorem to give a unified treatment of three major results
in ergodic theory: Sinai's factor theorem \cite{Sinai}, Ornstein's
isomorphism theorem \cite{Ornstein}, and Krieger's generator theorem
\cite{Krieger1}, which states any given any nonatomic invertible ergodic
measure-preserving system with finite entropy less than $\log N$, the space
$[N]^{\mathbb{Z}}$ can be endowed with a shift-invariant measure that
makes it
isomorphic to the given system.

Let $(\mathrm{X}, \mu, T)$ and $(\mathrm{Y}, \nu, S)$ be
measure-preserving systems.
A \textit{coupling} of $\mu$ and $\nu$ is a measure (i.e., not
necessarily the product measure) on the product space
$\mathrm{X} \times\mathrm{Y}$ that has as its projections the measures
$\mu$ and $\nu$;
a coupling that is also invariant under $T \times S$
is a \textit{joining} \cite{Rue}.
The set of joinings is always nonempty because of the product measure,
and the set of joinings that are supported on a~subset of $\mathrm{X}
\times\mathrm{Y}$ that is a
graph, is exactly the set of factors! Burton and Rothstein's alternative
to explicitly constructing factors is to prove they form a residual (large)
subset in the set of joinings. Our proof of Theorem~\ref{result} will
take place in this setting.

A joining $\zeta$ of $\mu$ and $\nu$ is
\textit{monotone} if
\[
\zeta\bigl\{ (\mathrm{x},\mathrm{y}) \in[N]^{\mathbb{Z}} \times
[N]^{\mathbb{Z}} \dvtx \mathrm{x}_0 \geq\mathrm{y}_0
\bigr\}=1.
\]
If $\mu$ stochastically dominates $\nu$, then the space
of all ergodic monotone joinings of $\mu$ and $\nu$
is nonempty.
Our proof of Theorem~\ref{result} will proceed as follows. We will
give more precise definitions later; here we only try to give an idea
of the proof.
Given a monotone joining of $\mu$ and $\nu$, via a construction of
Ball, we will perturb it to another monotone joining with large amounts
of independence between blocks,
then via the coupling of del Junco, we will
perturb the resulting joining to obtain a~monotone joining that is
an $\varepsilon$-almost factor, one that is a factor except on a set of
measure less than $\varepsilon$. This will be the key ingredient that
will allow us to conclude in the weak-star topology (see Section~\ref{ws})
that the set of $\varepsilon$-almost factors is an open dense set for
every $\varepsilon$;
intersecting over all $\varepsilon$ and using the Baire category
theorem implies that the
resulting set is nonempty, and thus there exists a monotone factor.
Recently, we also used the Burton--Rothstein method to prove a
Krieger generator theorem for (nonhyperbolic) toral automorphisms \cite{QSb}.

\subsection{Finitary constructions}

Keane and Smorodinsky \cite{keanea,keaneb} strengthened the results
of Sinai and Ornstein by constructing factors that are
\textit{finitary};
that is, on a set of full measure the factors constructed by Keane and
Smorodinsky are continuous with respect to the product topology on the
space $[N]^{\mathbb{Z}}$ and thus have the property that for almost every
$\mathrm{x}\in[N]^\mathbb{Z}$, there is a $k$ such that if $x_i=x'_i$
for all $|i|\le k$,
then $\phi(x')_0=\phi(x)_0$. See also \cite{MR571401,MR2308235,MR633760,MR2306207,MR2876281,shea} for background
and recent developments with regards
to finitary factors. Let us also note that Ball's monotone factor
is also finitary \cite{Ball}, but the factor we construct will
\textit{not} be. It will be interesting to see if the construction in
\cite{MR2308235}
can be adapted to give monotone factors, since their construction in the
case where there is a strict entropy gap, $H(p) > H(q)$, has a
coding radius with exponential tails, so that the probability
that $k$ of the coordinates of $\mathrm{x}$ are insufficient to determine
the zeroth coordinate of the image decays to zero exponentially
fast as $k \to\infty$.

\begin{question}
Is Theorem~\ref{result} true with the additional requirement
that the factor be finitary?
\end{question}

\subsection{Unilateral constructions}
Sinai's original theorem also applies in the case where the original
nonatomic ergodic measure-preserving system $(\mathrm{X}, \mu, T)$ is
not invertible,
in which case,
any one-sided Bernoulli shift on $[N]^{\mathbb{N}}$ of no greater
entropy can be obtained
as a factor of $(\mathrm{X}, \mu, T)$. In particular, for the case of
Bernoulli shifts,
Sinai defined factor maps that are
\textit{unilateral} so that zeroth coordinate of the image of almost
every point $\mathrm{x} \in[N]^{\mathbb{Z}}$ depends only the future
coordinates
of $\mathrm{x}$, given by $(\mathrm{x}_i)_{i =0}^{\infty}$. Within the powerful
framework of Ornstein theory \cite{MR0447525}, Ornstein and Weiss \cite{MR0412386}
also
extended the one-sided version of the Sinai theorem to mixing Markov
chains with positive transitions, but their construction is not finitary;
see also the proof and extension\vadjust{\goodbreak} to all mixing Markov chains given by
Propp \cite{MR1164596}. In the case where both Bernoulli shifts
give nonzero mass to at least three symbols, del Junco \cite{Juncoa,Junco}
further strengthened the results of Keane and Smorodinsky by constructing
unilateral finitary factors and isomorphisms. Because del Junco was
interested in constructing unilateral factors, he defined what he called
the star-joining to replace the more combinatorial marriage theorem that
is used in the Keane and Smorodinsky proofs, but was not suitable for
the unilateral case.
Let us also note that the factor we construct will \textit{not} be unilateral.

\begin{question}
Is Theorem~\ref{result} true with the additional requirement
that the factor be unilateral?
\end{question}

\subsection{Point processes and monotone thinning}
Ornstein theory also extends to much more general spaces.
In particular, Ornstein and Weiss \cite{MR910005} proved that
any two (homogeneous) Poisson processes on $\mathbb{R}^{d}$ are isomorphic.
Note that a Poisson process on $\mathbb{R}^d$ is stochastically dominated
by a Poisson process on $\mathbb{R}^d$ of higher intensity, and given a
Poisson process on $\mathbb{R}^d$ selecting each point independently with
some probability fixed probability gives a Poisson process of
lower intensity; sometimes this is referred to as independent
(randomized) thinning.
In the case $d=1$, Ball \cite{MR2133893} proved that any Poisson
process can be obtained as a monotone factor of a Poisson process
of higher intensity; that is, as a translation-equivariant
(nonrandomized) function of the higher intensity process, a set of
points is removed so that the remaining set forms a Poisson
process of lower intensity.
Holroyd,
Lyons, and Soo \cite{MR2884878} extended this result to all dimensions $d$.
%
For Poisson processes on a finite volume, Angel, Holroyd and Soo
\cite{MR2736350} proved a necessary and sufficient condition on the two
intensities
for the existence of a nonrandomized
thinning. 
See also \cite{GGP} for the related question of
nonrandomized thickening, \cite{Evans,tom} for cases where
nonrandomized equivariant thinning is impossible and \cite{mester}
for a case where even a monotone invariant coupling is impossible.


\section{Some tools used in the proof}


\subsection{Markers}
\label{markers}
Let $\zeta$ be a joining of the two Bernoulli measures
$\mu=p^{\mathbb{Z}}$ and $\nu=q^{\mathbb{Z}}$ on $[N]^{\mathbb{Z}}$.
Suppose that $p_a, p_b >0$, where
$0 \leq a <b < N$. Let $k_{\mathrm{mark}}$ be a
large positive integer that we will fix later. Given $\mathrm{x} \in
[N]^{\mathbb{Z}} $,
for $n \in\mathbb{Z}$ we say that $[n,n+2k_{\mathrm{mark}}] \subset
\mathbb{Z}$ is a
\textit{marker}
if $\mathrm{x}_{n+i}=a$ for all $ 0 \leq i \leq 2k_{\mathrm{mark}}-1$
and $\mathrm{x}_{n+2k_{\mathrm{mark}}}=b$;
we call $n$ the \textit{left endpoint} and $n+ 2k_{\mathrm{mark}}$ the
\textit{right endpoint}.
Note that markers have been defined so that no two markers will intersect.
%


\subsection{The quantile coupling}
The law of a random variable $X$ is the measure given by $\mathbb{P}(X
\in\cdot)$, and
if $X$ is real-valued, its distribution function is given by $F(x)=F_X(x):=
\mathbb{P}(X \leq x)$ for all $x \in\mathbb{R}$.
The generalized inverse of a distribution function is given by $F^{-1}(y):=
\sup\{x\dvtx  F(x) < y\}$. If two random variables $X$~and~$Y$ have the
same law, then we write $X \stackrel{d}{=} Y$.

In probabilistic terms, a coupling
of two random variables $X$ and $Y$ is a pair of random variables
$X'$ and $Y'$ defined on the same probability space such that
$X' \stackrel{d}{=} X$ and $Y' \stackrel{d}{=} Y$.
If $X$ and $Y$ take values in finite sets
$A$ and $B$, then an element $x \in A$ is \textit{split} by the
coupling if there exist
distinct $y, z \in B$ such that $\mathbb{P}(X' =x, Y'=y) >0$ and
$\mathbb{P}(X' =x, Y'=z) >0$; given a subset $B' \subset B$, we say
that $x \in A$
is \textit{split in} $B'$ if there exist distinct $y, z \in B'$ such that
$\mathbb{P}(X' =x, Y'=y) >0$ and
$\mathbb{P}(X' =x, Y'=z) >0$. For a probability measure $\alpha$ on $A
\times B$,
we define splitting in a similar way.

The quantile coupling is defined in the following way. Let $X$ and $Y$
be two
real-valued random variables with distribution functions $F$ and $G$.
Let $U$ be uniformly distributed on the unit interval $[0,1]$. It is
easy to verify that
$X':=F^{-1}(U) \stackrel{d}{=} X$ and $Y':=G^{-1}(U) \stackrel{d}{=} Y$
and that if the law
of $X$ stochastically dominates the law of $Y$,
then $ X'\geq Y'$; see \cite{MR1741181}, Chapter~1, Section~3,  for details.

%
\begin{remark}
\label{useful}
A very useful property of the quantile coupling is that if $X$ and $Y$ take
values in finite sets $A$ and $B$, then under the quantile coupling at most
$\#B-1$ elements of $A$ are split.
\end{remark}
More generally, if $X$ is a random variable taking values in a totally ordered
complete space, then the distribution function $F(x) = \mathbb{P}(X\leq
x)$ and its
generalized inverse are well defined, so the quantile coupling applies.

\subsection{Marriage and coupling}
\label{hall}

Let $A$ and $B$ be finite sets. If $\alpha$ and $\alpha'$ are probability
measures on $A \times B$ such that
for all $x\in A$ and all $y\in B$ we have $\alpha(x,y)=0$
implies $\alpha'(x, y)=0$, then $\alpha'$ is absolutely continuous with
respect to $\alpha$, and we say that $\alpha'$ is
\textit{subordinate} to $\alpha$.

We will make use of the variation of Keane and
Smorodinsky's marriage theorem \cite{keanea}, Theorem~11,  stated in the
language of measures.

%
\begin{proposition}[(Keane and Smorodinsky)]
\label{sub}
Let $A$ and $B$ be finite sets. If $\alpha$~is a probability
measure on $A \times B$, then for all $B' \subset B$ there
exists a probability measure $\alpha'$ such that:
\begin{longlist}[(i)]
\item
\label{subordinate}
$\alpha'$ is subordinate to $\alpha$,
\item
\label{superposition}
$\alpha'(A, \cdot) = \alpha(A, \cdot)$ and
$\alpha'(\cdot, B) = \alpha(\cdot, B)$ and
\item
\label{refined}
$\alpha'$ splits at most $\#B'-1$ elements in $B'$.
\end{longlist}
\end{proposition}

The proposition follows immediately from Ball's variation (\cite{Ball},
Lemma~6.1)
of \cite{keanea}, Theorem~11, and \cite{Ball}, Lemma~3.2. For more information,
see \cite{MR525312}, Section~4 and \cite{MR1073173}, Section~6.5, in Karl
Petersen's textbook; in particular, see \cite{MR1073173}, Chapter~6, Lemma~5.13
for a discussion of the relation between \cite{keanea}, Theorem~11 and the usual
Hall marriage theorem \cite{Hall01011935}.

\begin{remark}
\label{coupling-super}
Note that in Proposition~\ref{sub} that if
$\alpha:=\mathbb{P}(X \in\cdot,  Y \in\cdot)$ is the joint
distribution of random
variables $X$ and $Y$ taking values in $A$ and $B$, respectively, then by
\eqref{superposition} the probability measure $\alpha'$ given by the
proposition is a coupling of $X$ and $Y$. Moreover, by \eqref{subordinate},
if $A$ and $B$ are subsets of a poset and $\alpha$ is a monotone coupling
of $X$ and $Y$, then so is $\alpha'$.
\end{remark}

\subsection{Weak-star metric}
\label{ws}

Let $N >0$. For $i\geq0$, let $\mathcal C_i$ be the set of measurable
$C \subset [N]^{\mathbb{Z}} \times[N]^{\mathbb{Z}}$ that only depend
on the coordinates
$j\in[-i,i]$, so that $\mathrm{z} \in C$ implies that $\mathrm{z}' \in
C$ if
$z_j = z'_j$ for all $j \in[-i,i]$. We define the
\textit{weak-star}
metric $d^{*}$ on the space of measures on $[N]^{\mathbb{Z}}
\times[N]^{\mathbb{Z}}$ by\vspace*{-3pt} setting
\[
d^{*} (\zeta, \xi):= \sum_{i=0}^ \infty
2^{-(i+1)} \sup_{C \in\mathcal C_i} \bigl|\zeta(C) - \xi(C)\bigr|.
\]
The metric $d^{*}$ generates the usual weak-star topology,
and convergence in this topology is equivalent to what is sometimes
referred to as \textit{weak} convergence in probability theory \cite{Rud}, (7.4),
\cite{Dudley}, Chapter~11.

\subsection{Ball's joining}
\label{ball}

For $A \subset\mathbb{Z}$ and $\mathrm{x} \in[N]^{\mathbb{Z}}$, we
let $\mathrm{x}|_A \in
[N]^A$ denote $\mathrm{x}$ restricted to the elements of $A$. Also let
$(\mathrm{x}, \mathrm{y})
|_A= (\mathrm{x}|_A, \mathrm{y}|_A)$.
For the measure $\zeta$ on $[N]^{\mathbb{Z}} \times[N]^{\mathbb{Z}} $
and any
$A \subset \mathbb{Z}$, we let $\zeta|_{A}$ denote
the measure $\zeta$ restricted to $[N]^A \times[N]^A$, so that
for all measurable $F \subset [N]^A \times[N]^A$, we have
$\zeta|_{A} (F) = \zeta(F')$, where
$F':= \{ (\mathrm{x}, \mathrm{y})\dvtx  (\mathrm{x}|_A, \mathrm{y}|_A )\in
F\}$.
Sometimes we will refer to $\zeta|_{A}$
simply as $\zeta$ restricted to $A$.

Ball \cite{Ball}, pages 214--215,  defines a joining of two
Bernoulli shifts that has certain useful independence properties.
Let $p$ and $q$ be probability measures on $[N]$, where $p$ stochastically
dominates $q$. Let $\varrho$ be the quantile (monotone) coupling of $p$
and $q$.
Let $\zeta$ be an arbitrary ergodic monotone joining of $p^{\mathbb
{Z}}$ and $q^{\mathbb{Z}}$.
Then let $\gamma$ be the monotone coupling of the finite product measures
$p^{k_{\mathrm{mark}}}$ and $q^{k_{\mathrm{mark}}}$ given by
$\gamma=\gamma_{\zeta}:= \zeta|_{[1,k_{\mathrm{mark}}]}$.
Here let
$k_{\mathrm{mark}}$, $a$, and $b$ be as in Section~\ref{markers}. We
define a~monotone coupling of $p^{\mathbb{N}}$ and $q^{\mathbb{N}}$ by
alternating\vspace*{1pt} between
$\gamma$ and $\varrho$ in the following way. If
$Z=({X}, {Y})$ has law $\gamma$ and
${X}=(a, \ldots, a)=a^{k_{\mathrm{mark}}}$,
or if $Z=(X,Y)$ has law $\varrho$ and
$X\neq a$, then we say that a \textit{switch} occurs.
Let $\dot{\zeta}$ be given by sampling from $\gamma$
independently until a switch occurs, afterwards, sample from
$\varrho$ until a switch occurs; by switching back\vspace*{1pt} and forth
between $\gamma$ and $\varrho$ we obtain a monotone coupling
$\dot{\zeta}$ of $p^{\mathbb{N}}$ and $q^{\mathbb{N}}$.

To see that $\dot{\zeta}$ is, in fact, a coupling of $p^{\mathbb{N}}$ and
$q^{\mathbb{N}}$, let $k,n \geq1$ and $\alpha$ be a coupling of $p^k$ and
$q^k$, and $\beta$ be a coupling of $p^n$ and $q^n$. Observe that
if $W:=(W_i)_{i \in\mathbb{N}}$ is an i.i.d. sequence of random variables
with law $\alpha$ and $(R_i)_{i \in\mathbb{N}}$ is an i.i.d. sequence of
random variables with law $\beta$, then for any finite deterministic
$\ell$ the random variable given by $Z_{\ell}:=(W_1, \ldots, W_{\ell},
R_1, R_2, \ldots)$ is a coupling of $p^{\mathbb{N}}$ and $q^{\mathbb
{N}}$. Furthermore,
if $L$ is a stopping time for $W$, so that for all positive integers
$\ell$
the event $\{L \leq\ell\}$ belongs to the sigma-algebra generated by
$(W_i)_{i=1}^{\ell}$, then it is also true that $Z_L$ is a coupling of
$p^{\mathbb{N}}$ and $q^{\mathbb{N}}$. Since the switches also are
stopping times,
the result follows from repeated applications of this\vadjust{\goodbreak} simple observation.

For $\dot{\zeta}$, since a marker consists of $2k_{\mathrm{mark}}$
$a$'s followed by a $b$, we see that
no matter where the marker starts relative to the switches,
where the $b$ occurs the $\varrho$ coupling is used and a switch occurs,
so that $\dot{\zeta}$ restricted to the following interval of size
$k_{\mathrm{mark}}$ will
always be obtained from the $\gamma$ coupling.\vspace*{1pt}
Recall that we defined $S$ to be the left-shift. By first
stationarizing $\dot{\zeta}$, by setting
$\ddot{\zeta}:= \lim_{n \to\infty} \frac{1}{n}
\sum_{i=1}^n \dot{\zeta} \circ S^{i}$,
where the limit is taken in the weak-star topology, and then taking the
natural extension (for details see, e.g., \cite{down}, Section~4.3)
of $\ddot{\zeta}$ to $[N]^{\mathbb{Z}} \times[N]^{\mathbb{Z}}$, we
obtain a monotone ergodic joining of $p^{\mathbb{Z}}$ and $q^{\mathbb
{Z}}$, which we
denote by ${\zeta}_{\mathrm{alt}}$ and refer to as the \textit
{alternating} joining.
By the above observation, once we see a marker,
then we can determine (using the $\mathrm{x}$ variable alone)
which of $\varrho$ and $\gamma$ is being used for all coordinates to
the right.
Since with probability one, there are markers to the left of any point,
we see that
almost surely we can, by looking at the $\mathrm{x}$ variable, decide
which of $\varrho$ and
$\gamma$ is being used at each coordinate.
In Ball's paper, the coupling $\gamma$ is defined to satisfy additional
properties that she needs for her argument, but are not needed here.

The joining ${\zeta}_{\mathrm{alt}}$ has the following property.
For a given $\mathrm{x} \in[N]^{\mathbb{Z}}$ we define a bi-infinite
sequence of
\textit{alternating} intervals $\mathbf{K}(\mathrm{x})=(I_i)_{i \in
\mathbb{Z}}$
that partition $\mathbb{Z}$ into
intervals of length $k_{\mathrm{mark}}$ and $1$ in the following way. Locate
all the markers of $\mathrm{x}$. Any $n\in\mathbb{Z}$ that belongs to
the right endpoint
of a marker is an interval of length $1$, following a marker
will always be an interval of length $k_{\mathrm{mark}}$, and if
$\mathrm{x}$ restricted
to the interval of length $k_{\mathrm{mark}}$ is not a string of
$k_{\mathrm{mark}}$
consecutive $a$'s, then the following interval will also be one of
length $k_{\mathrm{mark}}$, otherwise, the following intervals will all
be of
length $1$, until a symbol that is not $a$ occurs;
the following interval will be one of length $k_{\mathrm{mark}}$.

Let $\Gamma= \Gamma_{\zeta}$ be the measure $\zeta
|_{[1,k_{\mathrm{mark}}]}$
conditioned so that a switch does not occur. A random variable with law
$\Gamma$ takes values in $[N]^{k_{\mathrm{mark}}} \times[N]^{k_{\mathrm
{mark}}}$.
For an interval $I \subset\mathbb{Z}$ of size $k$, we will often make
the identification $[N]^I \equiv[N]^k$.

%
\begin{proposition}[(Ball)]
\label{alt-ball}
Let $\zeta$ be an ergodic monotone joining of two Bernoulli
measures $\mu$ and $\nu$.
The alternating joining ${\zeta}_{\mathrm{alt}}$ is another ergodic
monotone joining of $\mu$
and $\nu$. If $\mathrm{Z}=(\mathrm{X}, \mathrm{Y})$ has law ${\zeta
}_{\mathrm{alt}}$, then conditional on the
alternating intervals $\mathbf{K}(\mathrm{X}) = (I_i)_{i \in\mathbb{Z}}$,
the random variable $\mathrm{Z}$ has the following properties:
\begin{itemize}
\item
The random variables $(\mathrm{Z}|_{I_i})_{i \in\mathbb{Z}}$ are independent.
\item
On each alternating interval $I$ of size $1$ not immediately
to the left of an interval
of size $k_{\mathrm{mark}}$, the law of $\mathrm{Z}|_I=(a, \mathrm
{Y}|_I)$ is $\varrho$ conditioned
on the event a switch does not occur, otherwise the law of $\mathrm
{Z}|_I$ is
$\varrho$ conditioned on the event that a switch occurs.
\item
On each alternating interval $I$ of size $k_{\mathrm{mark}}$ that is
not immediately left of an
interval of size $1$, the law of $\mathrm{Z}|_I$ is $\Gamma$ (a switch
does not occur); otherwise it is $\gamma$ conditioned so that a switch
does occur.
\end{itemize}
\end{proposition}

\begin{pf}
The result follows from the definition of ${\zeta}_{\mathrm{alt}}$.
\end{pf}

Let $\zeta$ be a monotone joining of Bernoulli measures. Let $k_{\mathrm
{mark}} >0$,
and ${\zeta}_{\mathrm{alt}}$ be the associated alternating joining. Let
$(\mathrm{x},\mathrm{y})$ be in the
support of ${\zeta}_{\mathrm{alt}}$. For each $n \in\mathbb{Z}$,
we say $n$ is \textit{frozen} if $n$ belongs an alternating interval in
$\mathbf K(\mathrm{x})$
of size $1$
or an alternating interval of size $k_{\mathrm{mark}}$ where a switch
occurs. Similarly,
we say that any alternating interval of size $1$ is \textit{frozen} and
any alternating
interval where a switch occurs is \textit{frozen}. We say that any
coordinate or
alternating interval that is not frozen is \textit{free}.

%
\begin{lemma}
\label{frozen}
Let $\zeta$ be a monotone joining of two Bernoulli measures $\mu=
p^{\mathbb{Z}}$ and $\nu= q^{\mathbb{Z}}$.
Given $k_{\mathrm{mark}}$, let ${\zeta}_{\mathrm{alt}}$ be the
associated alternating joining.
For $k_{\mathrm{mark}}$ sufficiently large the
probability that an integer $n \in\mathbb{Z}$ belongs to a frozen
interval can be
made arbitrary small.
\end{lemma}

\begin{pf}
Note that if the origin is in an alternating interval of size
$k_{\mathrm{mark}}$,
then the probability that this interval is a switch is exactly
$p_a^{k_{\mathrm{mark}}}$,
which goes to zero as $k_{\mathrm{mark}} \to\infty$.
A simple calculation will show that the probability that origin is in an
alternating interval of size one can be made arbitrarily small.
Let $F_i$ be the event that $i \in\mathbb{Z}$ is an alternating
interval of size $1$.
Note that $\mathbb{P}(F_0) = \mathbb{P}(F_1)$. We have that
\begin{eqnarray*}
\mathbb{P}(F_1) &=& \mathbb{P}(F_1 | F_0)
\mathbb{P}(F_0) + \mathbb{P}\bigl(F_1 |
F_0^c\bigr)\mathbb{P}\bigl(F_0^c
\bigr)
\\
&=& \mathbb{P}(F_1 | F_0)\mathbb{P}(F_1)
+ \mathbb{P}\bigl(F_1 | F_0^c\bigr)
\mathbb{P}\bigl(F_0^c\bigr)
\\
& \leq& (1-p_a)\mathbb{P}(F_1) + p_a^{k_{\mathrm{mark}}}
\mathbb{P}\bigl(F_0^c\bigr);
\end{eqnarray*}
thus $\mathbb{P}(F_1) \leq p_a^{k_{\mathrm{mark}}-1}$.
\end{pf}

%
\begin{lemma}
\label{weak-star-ball}
Let $\zeta$ be a monotone joining of two Bernoulli measures $\mu=
p^{\mathbb{Z}}$ and $\nu= q^{\mathbb{Z}}$. For any $\varepsilon>0$,
there exists
$k_{\mathrm{mark}}$ sufficiently large so that \mbox{$d^{*}(\zeta, {\zeta
}_{\mathrm{alt}}) < \varepsilon$}.
\end{lemma}

\begin{pf}
It suffices to show that for any integer $n>0$ and $\varepsilon>0$,
there exists
a~$k_{\mathrm{mark}}$ sufficiently large such that $|\zeta(C) - {\zeta
}_{\mathrm{alt}}(C)| < \varepsilon$
for all $C \in\mathcal{C}_n$.

Let $\mathrm{Z}$ and $\mathrm{Z}'$ be random variables with laws $\zeta
$ and ${\zeta}_{\mathrm{alt}}$,
respectively. Take $k_{\mathrm{mark}} >2n+1$. Let $G$ be the event
(measurable with
respect to $\mathrm{Z}'$) such that the interval $[-n, n]$ is contained
in an
alternating interval of size $k_{\mathrm{mark}}$, and $G^c$ denote the
complement. We have
\[
\mathbb{P}\bigl(\mathrm{Z}' \in C\bigr) = \mathbb{P}\bigl(
\mathrm{Z}' \in C  |  G\bigr)\mathbb{P}(G) + \mathbb{P}\bigl(
\mathrm{Z}' \in C  |  G^c\bigr)\mathbb{P}
\bigl(G^c\bigr),
\]
for all $C \in\mathcal{C}_n$.
By Proposition~\ref{alt-ball}, $\mathbb{P}(\mathrm{Z}' \in C  |  G) =
\mathbb{P}(\mathrm{Z} \in C)$.
By Lemma~\ref{frozen}, we can also choose $k_{\mathrm{mark}}$ so that
$\mathbb{P}(\mathrm{Z}' \in G) > 1 - \varepsilon/2$.
\end{pf}

\subsection{The Shannon--McMillan--Breiman theorem}
\label{SMB}
The Shannon--McMillan--Breiman theorem \cite{MR0092710} states that for
an ergodic\vspace*{1pt} invariant
measure $\mu$ on $[N]^{\mathbb{Z}}$ with entropy $h(\mu)$, for $\mu$-almost every
$\mathrm{x} \in[N]^{\mathbb{Z}}$, we have
\[
-\lim_{n \to\infty}\frac{1}{n} \log\mu\bigl(C_n(
\mathrm{x})\bigr)= h(\mu),
\]
where $C_n(\mathrm{x}):= \{\mathrm{x}' \in[N]^{\mathbb{Z}}\dvtx \mathrm{x}
|_{[0,n)} =
\mathrm{x}' |_{[0,n)}\}$.

Let $\zeta$ be an ergodic monotone joining of Bernoulli measures.
Recall that $\Gamma$ was $\zeta
|_{[1,k_{\mathrm{mark}}]}$
conditioned\vspace*{1pt} so that a switch does not occur.
Consider the identification $[N]^{k_{\mathrm{mark}}}
\equiv[N^{k_{\mathrm{mark}}}]$, and
the
measure ${\zeta}_{\mathrm{fill}}$ on $[N^{k_{\mathrm{mark}}}]^{\mathbb
{Z}} \times[N^{k_{\mathrm{mark}}}]^{\mathbb{Z}}$
given by $\Gamma_{\zeta}^{\mathbb{Z}}$.
Let ${\mu}_{\mathrm{fill}}$ and ${\nu}_{\mathrm{fill}}$ be the
respective projections of
${\zeta}_{\mathrm{fill}}$. Note that ${\mu}_{\mathrm{fill}}$ and ${\nu
}_{\mathrm{fill}}$
are Bernoulli measures.\vspace*{-1pt}

%
\begin{lemma}
\label{mono-entropy}
Let $\zeta$ be a monotone joining of two Bernoulli measures $\mu
=p^{\mathbb{Z}}$ and $\nu=q^{\mathbb{Z}}$.
Suppose that $H(p) > H(q)$. For $k_{\mathrm{mark}}$ sufficiently large,
we have
$h({\mu}_{\mathrm{fill}}) > h({\nu}_{\mathrm{fill}})$.\vspace*{-1pt}
\end{lemma}

The proof of Lemma~\ref{mono-entropy} follows from the following lemma,
the proof of which will involve some entropy calculations. If $X$ is a discrete
random variable taking values in a countable set $E=(e_i)_{i=1}^{\infty}$,
with\vspace*{1pt} probability distribution~$r$, we set $H(X)=H(r)$. Similarly, if
$\mathcal R=(R_i)_{i=1}^{\infty}$ is a partition of a probability space,
where $r_i = \mathbb{P}(R_i)$, then we also set $H(\mathcal{R})=H(r)$.
We also
let ${X}_{\mathrm{part}}= (\{X=e_i\})_{i=1}^{\infty}$, so that
$H({X}_{\mathrm{part}}) = H(X)$.
For $t\in[0,1]$, let $\Phi(t)=-t\log t-(1-t)\log(1-t)$, the entropy
of a two-element partition with elements of size $t$ and $1-t$.

%
\begin{lemma}
\label{cap-H}
Let $Z=(X,Y)$ be a jointly distributed pair of random variables,
each taking values in a finite
set $E$. Let $e^*\in E$.
Let $\tilde X$ be the variable $X$ conditioned on $\{X\ne e^*\}$ and
$\tilde Y$ be
the variable $Y$ conditioned on $\{X\ne e^*\}$. Let $u:=\mathbb P(X=e^*)$.
Then:
\begin{itemize}
\item
$H(\tilde X)\ge H(X)-\Phi(u)$.
\item
$H(\tilde Y)\le H(Y)+\Phi(u)+u\log(\#E)$.
\end{itemize}
\end{lemma}

\begin{pf}
We may assume by relabeling that $E=\{1,2,\ldots,M\}$
and that \mbox{$e^*=M$}.
Let $r_i=\mathbb P(X=i)$. Then for the first inequality, we have
\begin{eqnarray*}
H(\tilde X)&= & -\sum_{i=1}^{M-1}
\frac{r_i}{1-r_M}\log\biggl(\frac{r_i}{1-r_M} \biggr)
\\[-2pt]
&= & -\frac{1}{1-r_M}\sum_{i=1}^{M-1}
r_i\log r_i + \log(1-r_M)
\\[-2pt]
&=& \frac{1}{1-r_M} \bigl(H(X) + r_M\log r_M +
(1-r_M)\log(1-r_M) \bigr)
\\[-2pt]
&\ge &  H(X)-\Phi(u).
\end{eqnarray*}

For the second inequality, let $\tilde Y'$ be an independent copy of
$\tilde Y$, and
define
\[
W=\cases{ Y,&\quad\mbox{if $X\ne e^*$,}
\cr
\tilde Y',& \quad\mbox{otherwise}.}
\]
Clearly $W$ has the same distribution as $\tilde Y$.
Let $\mathcal Q$ be the partition of the probability space into the two sets
$\{X=e^*\}$ and $\{X\ne e^*\}$.

We then have
\begin{eqnarray*}
H(\tilde Y)&=& H(W) \le H({W}_{\mathrm{part}}\vee{Y}_{\mathrm{part}}\vee
\mathcal
Q)
\\
&= & H({W}_{\mathrm{part}}  |  {Y}_{\mathrm{part}}\vee\mathcal
Q)+H({Y}_{\mathrm{part}}
\vee\mathcal Q)
\\
&\le &  H({W}_{\mathrm{part}} | {Y}_{\mathrm{part}} \vee\mathcal Q)+H(Y)+H(
\mathcal Q).
\end{eqnarray*}
By definition, $H(\mathcal Q)=\Phi(u)$.

If $X\ne e^*$ (an event with probability $1-u$), then knowing
in which element of ${Y}_{\mathrm{part}}\vee\mathcal Q$ a point lies,
determines $W$ and hence in which element of ${W}_{\mathrm{part}}$ it lies.
Otherwise, on a set of measure $u$, we simply know that
$W$ takes values in $E$. Hence $H({W}_{\mathrm{part}}  |  {Y}_{\mathrm
{part}}\vee\mathcal Q)$, which is the
expected amount of additional information gained by knowing
${W}_{\mathrm{part}}$ when
${Y}_{\mathrm{part}}\vee\mathcal Q$ is already known is at most $u\log
(\#E)$.
\end{pf}

\begin{pf*}{Proof of Lemma~\ref{mono-entropy}}
Let $\zeta$ be a joining of two Bernoulli measures $\mu$ and~$\nu$ on
$[N]^{\mathbb{Z}}$, and let $(X,Y)$ have law $\zeta
|_{[1,k_{\mathrm{mark}}]}=\gamma$. Note that $H(X)= k_{\mathrm{mark}} H(p)$
and $H(Y) = k_{\mathrm{mark}} H(q)$. Let $(\tilde X,\tilde Y)$ have law
$\Gamma$;
that is, $\gamma$ conditioned on the event that $X$ is not a string
of $k_{\mathrm{mark}}$ consecutive $a$'s; note that $\mathbb{P}(X =
a^{k_{\mathrm{mark}}})=p_a^{k_{\mathrm{mark}}}$.
Thus with Lemma~\ref{cap-H} we have
%
\begin{eqnarray}
 h(\mu_{\mathrm{fill}})-h(\nu_{\mathrm{fill}})&=& H(\tilde X)-H(
\tilde Y)
\nonumber
\\
\label{exp}&\ge & \bigl(H(X)-H(Y) \bigr)- \bigl(2\Phi\bigl(p_a^{k_{\mathrm{mark}}}
\bigr) - p_a^{k_{\mathrm{mark}}}\log N^{k_{\mathrm{mark}}} \bigr)
\\
&=& k_{\mathrm{mark}} \bigl(H(p) - H(q) \bigr) - \bigl(2\Phi\bigl
(p_a^{k_{\mathrm{mark}}}
\bigr) -k_{\mathrm{mark}} p_a^{k_{\mathrm{mark}}}\log N \bigr).\nonumber
\end{eqnarray}
Since we assume that $H(p) > H(q)$, the first term on the right-hand
side of
\eqref{exp} grows linearly as a function of $k_{\mathrm{mark}}$,
whereas the second term
decreases to zero exponentially as a function $k_{\mathrm{mark}}$.
\end{pf*}
Note that in our proof of Lemma~\ref{mono-entropy}, we made use of
the strictness of the inequality $H(p) >H(q)$.

%


\subsection{The star-coupling}
\label{delstar}

Let $(X_1,Y_1)$ and $(X_2, Y_2)$ be finite valued random variables
taking values in $(E_1, F_1)$ and $(E_2, F_2)$, where $E_1$ and $F_2$
are totally ordered via $<_1$ and $<_2$.
Following del Junco \cite{Juncoa,Junco}, we define the \textit{star-coupling}
of $(X_1,Y_1)$ and $(X_2, Y_2)$ in the following way. Set
$s_{f_1}(e_1):= \mathbb{P}(X_1 \leq_1 e_1  |  Y_1=f_1)$ and
$t_{e_2}(f_2):= \mathbb{P}(Y_2 \leq_2 f_2  |  X_2=e_2)$.
Let $V_2$, $V_1$
and $U$ be independent random variables uniformly distributed in
$[0,1]$. Set
\[
X_2':= F_{X_2}^{-1}(V_2)
\quad \mbox{and}\quad  Y_1':=F_{Y_1}^{-1}(V_1),
\]
so that $X_2'$ and $Y_1'$ are independently sampled copies of $X_2$ and $Y_1$.
For all $e_2 \in E_2$ and $f_1 \in F_1$, if $X_2' = e_2$ and
$Y_1'=f_1$, then
we define $Y_2'$ and $X_1'$ via the following conditional quantile coupling:
\[
Y_2':= t_{e_2}^{-1}(U) \quad\mbox{and}\quad
X_1':= s_{f_1}^{-1}(U).
\]
%
Clearly $(X_1', Y_1') \stackrel{d}{=} (X_1, Y_1)$
and $(X_2', Y_2') \stackrel{d}{=} (X_2, Y_2)$.

%
\begin{remark}
\label{indep}
In the star-coupling of $(X_1, Y_1)$ and $(X_2, Y_2)$, $X_2'$ is
independent of $(X_1', Y_1')$ and $Y_1'$ is independent of $(X_2', Y_2')$.
\end{remark}
%

\begin{remark}
\label{usefulb}
It follows from Remark~\ref{useful} that the star-coupling of the
random variables $(X_1, Y_1)$ and $(X_2, Y_2)$ taking values on
$(E_1, F_1)$ and $(E_2, F_2)$, respectively, has the property that
for a fixed $e_2 \in E_2$ and $f_1 \in F_1$, the number of $e_1 \in E_1$
such that there are distinct $f_2, h_2 \in F_2$ with both $(e_1,f_1, e_2,f_2)$
and $(e_1,f_1, e_2,h_2)$ receiving positive mass under the star-coupling
$(X_1', Y_1',\break  X_2', Y_2')$ is at most $\#F_2 -1$.
\end{remark}

Remark~\ref{usefulb}, Proposition~\ref{sub}, and the
Shannon--McMillan--Breiman theorem
\cite{MR0092710} lead to the following useful modification of a proposition
of del Junco \cite{Junco}, Proposition~4.8.

Let $Z_i:= (X_i, Y_i)$ be finite valued random variables where each
of the $X_i$'s and $Y_i$'s take values on ordered spaces $E_i$ and $F_i$.
We define the \textit{iterative} star-coupling of $Z_1, \ldots, Z_n$ to
be a
random variable $W_n$ taking values on $(E_1 \times\cdots\times E_n)
\times(F_1 \times\cdots\times F_n)$ in the following way for the
case $n=3$; the definition for general $n$ will follow inductively.
Let $Z_1':= (X_1', Y_1')$ and $Z_2':=(X_2', Y_2')$ be the star-coupling of
$Z_1$ and $Z_2$. Set $W_2:= ((X_1', X_2'), (Y_1', Y_2') )$.
Note that $(X_1', X_2')$ takes values in the space
$E_1 \times E_2$, which we endow with the lexicographic ordering.
Now let the star-coupling of $W_2$ and $Z_3$ be
given by $W_2':= ((X_1^{\prime\prime}, X_2^{\prime\prime}),
(Y_1^{\prime\prime}, Y_2^{\prime\prime}) )$ and $Z_3':= (X_3', Y_3')$.
Set $W_3:= ((X_1^{\prime\prime}, X_2^{\prime\prime}, X_3'),
(Y_1^{\prime\prime}, Y_2^{\prime\prime}, Y_3') )$.

%
\begin{remark}
Note that in general, even if the star-coupling of $(X_1, X_2)$ and
$(Y_1, Y_2)$ is defined, the star-coupling of $(X_2,Y_2)$ and $(X_1, Y_1)$
may not be defined, since the required spaces may not be ordered, and even
if they are, there is a lack of commutativity. Note the iterative
star-coupling is defined in a certain order, so that the iterative
star-coupling of $Z_1, Z_2, Z_3$ is given by the star-coupling of the
star-coupling of $(Z_1,Z_2)$ and $Z_3$. It is possible to define the
star-coupling so that it is associative \cite{Junco}, Lemma~4.3; this
observation is important for del Junco's construction of isomorphisms,
but will not be important for us.
\end{remark}

%
\begin{proposition}[(del Junco)]
\label{delJunco}
Let $\mu=p^{\mathbb{Z}}$ and $\nu=q^{\mathbb{Z}}$ be Bernoulli measures
on $[N]^{\mathbb{Z}}$,
where $H(p) > H(q)$ and $p \succeq q$.
Let $\zeta$ be an ergodic monotone joining
of $\mu$ and $\nu$. Given a sufficiently large integer
$k_{\mathrm{mark}} \in\mathbb{Z}^{+}$ so that the conclusion of
Lemma~\ref{mono-entropy} holds,
and $\eta>0$,
there exists a $n_{\mathrm{initial}} \in\mathbb{Z}^{+}$, and random
variable~$\bar{Z}_0$
taking values on $[N]^ {n_{\mathrm{initial}} k_{\mathrm{mark}}} \times
[N]^ {n_{\mathrm{initial}} k_{\mathrm{mark}}}$ with law ${\beta
}_{\mathrm{sub}}$ that is subordinate
to the measure $\Gamma^{n_{\mathrm{initial}}}$ and has the same marginals,
such that for all $n>0$, the following holds.

Define $(Z_i)_{i=1}^n$ to be independent random variables with law
$\Gamma$.
Define the following product space:
\[
\mathbf{I}_j := [N]^ {n_{\mathrm{initial}} k_{\mathrm{mark}}} \times
\prod
_{i=1}^j [N]^{k_{\mathrm{mark}}} \equiv
\bigl[N^k_{\mathrm{mark}}\bigr]^{n_{\mathrm{initial}}+j}.
\]
%
Let $\mathbf{W}_n= (\mathbf{X}_n, \mathbf{Y}_n)$ be a
random variable given by the iterative
star-coupling of $\bar{Z}_0, Z_1, \ldots, Z_n$.
There exists a deterministic function
$\Psi\dvtx \mathbf{I}_n \to\mathbf{I}_n$ such that $\mathbb{P} (\mathbf
{Y}_n =
\Psi(\mathbf{X}_n) ) > 1 - \eta$.
\end{proposition}

The key feature of this proposition is that $n$ can be taken arbitrarily
large, independently of $k_{\mathrm{mark}}$ and $\eta$. When appealing
to Proposition~\ref{delJunco}, we will refer to the set $[N]^{n_{\mathrm
{initial}} k_{\mathrm{mark}}}
\times[N]^{n_{\mathrm{initial}} k_{\mathrm{mark}}}$ as
the \textit{initial block}.

\begin{pf*}{Proof of Proposition~\ref{delJunco}}
The proof is an adaptation of
\cite{Junco}, Proposition~4.7. We will place conditions on $n_{\mathrm{initial}}$ later. Set
\[
\mathbf{L}_j:= n_{\mathrm{initial}} + j\qquad \mbox{for } 0 \leq j \leq n.
\]
%
%
Recall that we assumed that $k_{\mathrm{mark}}$ was chosen to ensure that
$h_{\mathrm{gap}}:=h({\mu}_{\mathrm{fill}}) - h({\nu}_{\mathrm
{fill}})>0$. Let $\varepsilon\in
(0, h_{\mathrm{gap}}/2)$ and
\begin{equation}
\label{delta}
\delta:=h_{\mathrm{gap}}-2\varepsilon>0.
\end{equation}
Let $\mathbf{x} \in{\mathbf{I}_j}$ be given by
$\mathbf{x}= (x_{0}, \ldots, x_{j})$. We say that $\mathbf{x}$ is
${\mu}_{\mathrm{fill}}$-\textit{good} if
\begin{equation}
\label{pgood}
\mu_{\mathrm{fill}}|_{[1, \mathbf{L}_j]}(\mathbf{x}) <
e^{-(h(\mu_{\mathrm{fill}}) - \varepsilon) \mathbf{L}_j},
\end{equation}
and is
${\mu}_{\mathrm{fill}}$-\textit{completely good} if for all $ 0 \leq i
\leq j$, we have
$(x_{0}, \ldots, x_{i}) \in\mathbf{I}_i$ is good. Similarly,
for $\mathbf{y}= (y_{0}, \ldots, y_{j})
\in {\mathbf{I}_j}$, we say that $\mathbf{y}$ is ${\nu}_{\mathrm
{fill}}$-\textit{good} if
\begin{equation}
\label{qgood} \nu_{\mathrm{fill}}|_{[1,\mathbf{{L}}_j]}(\mathbf{y}) >
e^{-(h(\nu_{\mathrm{fill}}) +\varepsilon) \mathbf{L}_j},
\end{equation}
and is
${\nu}_{\mathrm{fill}}$-\textit{completely good} if for all $ 0 \leq i
\leq j $, we have $(y_{0},
\ldots, y_{i}) \in{\mathbf{I}_i} $ is good.

Let $\mathbf{I}_0({\nu}_{\mathrm{fill}})^{\mathrm{good}}$ denote the
set of
${\nu}_{\mathrm{fill}}$-good elements of $\mathbf{I}_0$. Note that
\[
\# \mathbf{I}_0({\nu}_{\mathrm{fill}})^{\mathrm{good}} \leq
e^{(h(\nu_{\mathrm{fill}}) +\varepsilon) \mathbf{L}_0}.
\]
By Proposition~\ref{sub}, there exists a random variable $\bar{Z}_0$
taking values on\break\break $[N]^ {n_{\mathrm{initial}} k_{\mathrm{mark}}} \!\times
[N]^ {n_{\mathrm{initial}} k_{\mathrm{mark}}}$
with
law ${\beta}_{\mathrm{sub}}$, that is:
\begin{itemize}
\item
subordinate
to $\Gamma^{n_{\mathrm{initial}}}$,
\item

has the same marginals as $\Gamma^{n_{\mathrm{initial}}}$ and
\item
where at most $e^{(h(\nu_{\mathrm{fill}}) +\varepsilon) \mathbf
{L_0}}-1$ elements of
$\mathbf{I}_0$ are split in $\mathbf{I}_0({\nu}_{\mathrm
{fill}})^{\mathrm{good}}$.
\end{itemize}
Let
\[
\mathbf{J}_j:=\mathbf{I}_0({\nu}_{\mathrm{fill}})^{\mathrm{good}}
\times\prod_{i=1}^j [N]^{k_{\mathrm{mark}}}.
\]

For $j \geq0$, let $\mathbf{W}_j=(\mathbf{X}_j, \mathbf{Y}_j)$
be a random variable
given by the iterative star-coupling of $\bar{Z}_{0}, Z_1,
\ldots, Z_{j}$, where we set $\mathbf{W}_0:=\bar{Z}_0$; thus
$\mathbf{X}_j$ and $\mathbf{Y}_j$
take values in $ {\mathbf{I}_j}$.
We say that $\mathbf{x} \in {\mathbf{I}_j}$ is
\textit{desirable} if the following properties are satisfied:
\begin{longlist}[(a)]
\item[(a)]
The element $\mathbf{x}$ is ${\mu}_{\mathrm{fill}}$-completely good.
\item[(b)]
The element $\mathbf{x}$ is not split in $\mathbf{J}_j$ by
$\mathbf{W}_j=(\mathbf{X}_j, \mathbf{Y}_j)$.
\item[(c)]
Furthermore, there exists a unique ${\nu}_{\mathrm{fill}}$-completely good
$\mathbf{y} \in {\mathbf{J}_j}$ for which $(\mathbf{x}, \mathbf{y})$
receives positive mass under $\mathbf{W}_j$.
\end{longlist}
For desirable $\mathbf{x} \in\mathbf{I}_j$, set $\Psi_{j}(\mathbf{x}) =
\mathbf{y}$, where $\mathbf{y}$ is determined by condition~(c);
otherwise if $\mathbf{x}$ is not desirable simply set
$\Psi_{j}(\mathbf{x})= \mathbf{y}'$ for the fixed
$\mathbf{y}'\in\mathbf{I}_j$ that is just a block of $0$'s. Note that
\[
\mathbb{P} \bigl(\mathbf{Y}_j =
\Psi_j(\mathbf{X}_j) \bigr) \geq\mathbb{P}(\mathbf{X}_j
\mbox{ is desirable}).
\]

Using Remark~\ref{usefulb}, we will use the inductive argument in the proof
of \cite{Junco}, Lemma~4.6,  to show that for all $j \geq0$,
\begin{eqnarray}
&&\mathbb{P}(\mathbf{X}_j \mbox{ is not desirable})
\nonumber
\\[-10pt]
\label{pre-SMB}
\\[-10pt]
&& \qquad \leq\mathbb{P}(\mathbf{X}_j \mbox{ is not c.g.}) + \mathbb{P}(
\mathbf{Y}_j \mbox{ is not c.g.}) +
 e^{-\delta\mathbf{L}_0}+ {N}^{k_{\mathrm{mark}}
}\sum_{i=0}
^ {j-1} e^{-\delta \mathbf{L}_i},\nonumber
\end{eqnarray}
where ``c.g.'' is short for completely good.

The case $j=0$ is easy, since being good implies being completely good,
and under $\bar{Z}_0$ at most $e^{(h(\nu_{\mathrm{fill}}) +\varepsilon)
\mathbf{L}_0}-1$
elements of $\mathbf{I}_0$ are split in $\mathbf{J}_0$; thus by
\eqref{pgood}, the ${\mu}_{\mathrm{fill}}$-measure of all the ${\mu
}_{\mathrm{fill}}$-good
elements that are split by $\mathbf{J}_0$ is at most
$e^{-(h(\mu_{\mathrm{fill}}) - \varepsilon) \mathbf{L}_0} \times
e^{(h(\nu_{\mathrm{fill}}) +\varepsilon)\mathbf{L}_0} \leq e^{-\delta
\mathbf{L}_0}$,
by \eqref{delta}.

Assume \eqref{pre-SMB} for the case $j-1 \geq0$. We show that
\eqref{pre-SMB} holds for the case $j$. Let $E$ be the event that
$\mathbf{X}_{j-1} \mbox{ is desirable but }
\mathbf{X}_{j} \mbox{ is not desirable}$. Clearly,
%
\begin{equation}
\label{cases}
\mathbb{P}(\mathbf{X}_j \mbox{ is not desirable})
\leq\mathbb{P}(\mathbf{X}_{j-1} \mbox{ is not desirable}) +
\mathbb{P}(E).
\end{equation}
Note that on the event $E$,
the random variables $\mathbf{X}_{j-1}$ and $\mathbf{Y}_{j-1}$ are
completely good.
Observe that the event $E$ is contained in the following three events:
\begin{longlist}[(III)]
\item[(I)]
$E_1:=$ The random variable $\mathbf{X}_j$ is not good, but $\mathbf{X}_{j-1}$
is completely good.
\item[(II)]
$E_2:=$ The random variable $\mathbf{X}_j$ is completely good, but is split
in $\mathbf{J}_j$ under the iterative star-coupling $\mathbf{W}_j$,
even though $\mathbf{X}_{j-1}$ is desirable.
\item[(III)]
$E_3:=$ The random variable $\mathbf{Y}_j$ is not good, but
$\mathbf{Y}_{j-1}$ is completely good.
\end{longlist}
Clearly,
\begin{equation}
\label{casea} \mathbb{P}(E_1) + \mathbb{P}(\mathbf{X}_{j-1}
\mbox{ is not c.g.}) = \mathbb{P}(\mathbf{X}_j \mbox{ is not c.g.}).
\end{equation}
Similarly,
\begin{equation}
\label{casec} \mathbb{P}(E_3) + \mathbb{P}(\mathbf{Y}_{j-1}
\mbox{ is not c.g.}) = \mathbb{P}(\mathbf{Y}_j \mbox{ is not c.g.}).
\end{equation}

Let us focus on the event $E_2$. Let $\mathbf{X}_{j} = ( \mathbf{X}_{j-1},X)$,
so that $X$ takes values in $[N]^{k_{\mathrm{mark}}}$. We show that for any
$x \in[N]^{k_{\mathrm{mark}}}$ and any completely good $\mathbf{y} \in
{\mathbf{J}_{j-1}}$ that
\begin{equation}
\label{use-usefulb}
\mathbb{P}(E_2 |  X=x, \mathbf{Y}_{j-1} =
\mathbf{y}) \leq N^{k_{\mathrm{mark}}
} e^{-\delta\mathbf{L}_{j-1}},
\end{equation}
so that $\mathbb{P}(E_2) \leq N^{k_{\mathrm{mark}}}e^{-\delta\mathbf
{L}_{j-1}}$ and it follows that
\eqref{pre-SMB} holds by \eqref{cases}, \eqref{casea}, \eqref{casec} and
the inductive hypothesis.

Note that if $\mathbf{x}$ and $\mathbf{y}$ are good, then
\begin{eqnarray}
\mathbb{P}( \mathbf{X}_{j-1} = \mathbf{x} |X=x, \mathbf{Y}_{j-1}
= \mathbf{y}) &=& \frac{\mathbb{P}(\mathbf{X}_{j-1} = \mathbf{x},
\mathbf{Y}_{j-1} =
\mathbf{y}, X=x)}{\mathbb{P}(\mathbf{Y}_{j-1} = \mathbf{y}, X=x)}
\nonumber
\\
\label{first} &=& \frac{\mathbb{P}(\mathbf{X}_{j-1} = \mathbf{x},
\mathbf{Y}_{j-1} = \mathbf{y})}{
\mathbb{P}(\mathbf{Y}_{j-1} = \mathbf{y})}
\\
&\leq& \frac{\mathbb{P}(\mathbf{X}_{j-1} = \mathbf{x})}{\mathbb
{P}(\mathbf{Y}_{j-1} =
\mathbf{y})}
\nonumber
\\
\label{division} & \leq& e^{-\delta\mathbf{L}_{j-1}},
\end{eqnarray}
where \eqref{first} follows from Remark~\ref{indep}
(with $X=X_2'$ and $\mathbf{Y}_{j-1} = Y_1'$) and \eqref{division}
follows from \eqref{pgood}, \eqref{qgood} and \eqref{delta}.
Also note that if $\mathbf{x}$ is desirable,
then $(\mathbf{x}, x)$ is split under $\mathbf{W}_j$ if and only if for
the unique $\mathbf{y}$ for which
$(\mathbf{x}, \mathbf{y})$ receives positive mass under $\mathbf{W}_{j-1}$
there exist distinct $y,y' \in[N]^{k_{\mathrm{mark}}}$
for which for which both $((\mathbf{x}, x), (\mathbf{y}, y))$ and
$((\mathbf{x}, x), (\mathbf{y}, y'))$
receive positive mass under $\mathbf{W}_j$.
By Remark~\ref{usefulb},
for a fixed $x \in[N]^{k_{\mathrm{mark}}}$ and $\mathbf{y} \in{\mathbf
{J}_{j-1}}$,
the set of all $\mathbf{x}$ such that there exists distinct $y,y'\in
[N]^{k_{\mathrm{mark}}}$
for which both $((\mathbf{x}, x), (\mathbf{y}, y))$ and $((\mathbf{x},
x), (\mathbf{y}, y'))$
receive positive mass under $\mathbf{W}_j$ has at
most $ N^{k_{\mathrm{mark}}}-1$ elements; thus summing over all such
$\mathbf{x}$
yields \eqref{use-usefulb}.

The Shannon--McMillan--Breiman theorem implies that $n_{\mathrm{initial}}$
can be chosen so that all
four terms in \eqref{pre-SMB} can be made smaller than $\eta/4$.
This is done in the following way.

Set
\[
S_{\mu}(k, K)
:=\bigl\{\mathbf{x} \in\bigl[N^{k_{\mathrm{mark}}}\bigr]^K\dvtx
{\mu}_{\mathrm{fill}}|_{[1,\ell]}(\mathbf{x}) < e^{-(h({\mu}_{\mathrm
{fill}})-\varepsilon)\ell} \mbox{ for all
} k \leq\ell\leq K\bigr\}
\]
and
\[
S_{\nu}(k, K)
:=\bigl\{\mathbf{y} \in\bigl[N^{k_{\mathrm{mark}}}\bigr]^K\dvtx
{\nu}_{\mathrm{fill}}|_{[1,\ell]}(\mathbf{y}) > e^{-(h({\nu}_{\mathrm
{fill}}) +\varepsilon)\ell} \mbox{ for all
} k \leq\ell\leq K\bigr\}.
\]
By the Shannon--McMillan--Breiman theorem choose $n_{\mathrm{initial}}$
so that for all
$K > n_{\mathrm{initial}}$, we have
\begin{eqnarray}
\label{SMB-top}
{\mu}_{\mathrm{fill}}| _{[1,K]}\bigl(S_{\mu}(n_{\mathrm
{initial}},
K)\bigr) &>& 1 - \eta/4 \quad \mbox{and}
\\
\label{SMB-bom}
{\nu}_{\mathrm{fill}}|_{[1,K]}\bigl(S_{\nu}(n_{\mathrm
{initial}},
K)\bigr) &>& 1 - \eta/4;
\end{eqnarray}
we can also require that
\begin{equation}
\label{pile}
N^{k_{\mathrm{mark}}}\sum_{i=n_{\mathrm{initial}}}^{\infty}
e^{-\delta i} < \eta/4.
\end{equation}


Conditions \eqref{SMB-top} and \eqref{SMB-bom} give that $\mathbb
{P}(\mathbf{X}_j
\mbox{ is not c.g.}) < \eta/4 $ and $\mathbb{P}(\mathbf{Y}_j \mbox{ is not}\break \mbox{c.g.}) <
\eta/4$, and \eqref{pile} ensures that
$ e^{-\delta \mathbf{L}_0} \leq\eta/4$, and ${N}^{k_{\mathrm
{mark}}}\sum_{i=0} ^ {j-1}
e^{-\delta \mathbf{L}_j} < \eta/4$; thus all four
terms on the right-hand side of inequality \eqref{pre-SMB} are less
than $\eta/4$.
\end{pf*}

%
\begin{remark}
In proof of Proposition~\ref{delJunco}, recall that we appealed to
Remark~\ref{usefulb} which is the reason for the term $N^{k_{\mathrm
{mark}}}$ in
\eqref{pre-SMB}. Since our proof of Theorem~\ref{delJunco} relies on
Proposition~\ref{delJunco}, we do not know if the analogue of
Theorem~\ref{result} is true if $q$ gives positive mass to a countable
number of symbols. Note the Sinai and Ornstein theorems include the
case where the entropy is possibly infinite and there are a countable
number of symbols. See, for example, \cite{down}, Section~4.5,
for a recent treatment.
\end{remark}

%
%
%
%

\section{Proof of Theorem~\texorpdfstring{\protect\ref{result}}{1}}

\subsection{The alternating star-joining}

Let $\zeta$ be a monotone joining of the two Bernoulli measures $\mu$
and $\nu$,
and let $k_{\mathrm{mark}}>0$, and ${\zeta}_{\mathrm{alt}}$ be its
associated alternating
joining. Assume that we have already applied Proposition~\ref{delJunco}
to obtain an $n_{\mathrm{initial}}$ and a probability measure $\beta
_{\mathrm{sub}}$ on
$[N]^ {n_{\mathrm{initial}} k_{\mathrm{mark}}} \times
[N]^ {n_{\mathrm{initial}} k_{\mathrm{mark}}}$
that is
subordinate
to $\Gamma^{n_{\mathrm{initial}}}$ and has the same marginals.
By re-sampling on (most of the) free intervals of ${\zeta}_{\mathrm
{alt}}$ by using the
star-coupling, we will produce another monotone joining ${\zeta
}_{\mathrm{alt}*}$ of
$\mu$ and $\nu$. We define ${\zeta}_{\mathrm{alt}*}$ in the following way.

Let $r_{\mathrm{mark}}$ be a large integer
to be chosen later. A
\textit{super marker}
is the maximal union of at least $r_{\mathrm{mark}}$ consecutive markers,
and we call
the set of integers between
and not including two super markers a \textit{large block}.

Let $\mathrm{Z}=(\mathrm{X}, \mathrm{Y})$ have law ${\zeta}_{\mathrm{alt}}$.
Call any large block with at least $n_{\mathrm{initial}}$ free
intervals an
\textit{action block}; we re-sample only on the action blocks. We
define a
new random variable $\mathrm{Z}'=(\mathrm{X}', \mathrm{Y}')$ taking
values on
$[N]^{\mathbb{Z}} \times[N]^{\mathbb{Z}}$ by first declaring that on every
frozen interval or free interval $I$ not belonging to an action block
that $\mathrm{Z}' |_I = \mathrm{Z}|_{I}$.
Next, for a fixed action block,
let $Z_0$ (taking values in $[N]^{k_{\mathrm{mark}}n_{\mathrm
{initial}}} \times
[N]^{k_{\mathrm{mark}}n_{\mathrm{initial}}}$) be
$\mathrm{Z}$ restricted to the first $n_{\mathrm{initial}}$ free
intervals. Let $(I_i)_{i=1}^n$
be the remaining free intervals, and let $Z_{i}=(X_i, Y_i)$ (taking values
in $[N]^k_{\mathrm{mark}} \times[N]^k_{\mathrm{mark}}$) be $\mathrm
{Z}$ restricted to $I_i$.
By Proposition~\ref{ball}, conditional on the alternating intervals
$\mathbf{K}(\mathrm{X})$, we have that $(Z_i)_{i=1}^n$ is an i.i.d.
sequence of
random variables with law $\Gamma_{\zeta}$, and the law of
${Z}_{0}=(X_0, Y_0)$
is given by the law of
$\Gamma_{\zeta}^{n_{\mathrm{initial}}}$, and is also independent of
$(Z_i)_{i=1}^n$.
Let $\bar{Z}_0=
(\bar{X}_0, \bar{Y_0})$ have law given by the measure
$\beta_{\mathrm{sub}}$.
Take the iterative star-coupling of the random variables
$\bar{Z}_0, Z_{1}, \ldots, Z_{n}$
to obtain a random variable
\[
\mathbf{W}= \bigl( (\bar{X}_0', X_1', \ldots, X_n'),   (\bar{Y}_0',Y_1',
\ldots, Y_n') \bigr),
\]
taking values on $[N]^{k_{\mathrm{mark}}n_{\mathrm{initial}}}
[N]^{k_{\mathrm{mark}} n}
 \times  [N]^{k_{\mathrm{mark}}n_{\mathrm{initial}}} [N]^{{k_{\mathrm
{mark}}}n}$,
which we call the \textit{star-filler} for an action block.

By Remark~\ref{indep}, independence of the $(Z_i)_{i=0}^n$ and the fact
that ${\beta}_{\mathrm{sub}}$ is a coupling of $X_0$ and $Y_0$, we have
\begin{eqnarray}
\label{short-remark}
({X}_{0}, X_{1}, \ldots,
X_{n}) &\stackrel{d} {=}& \bigl(\bar{X}'_{0},
X'_{1}, \ldots, X'_{n}\bigr)
\quad\mbox{and}
\nonumber
\\[-8pt]
\\[-8pt]
\nonumber
({Y}_{0}, Y_{1}, \ldots, Y_{n}) &\stackrel{d}{=}& \bigl(\bar{Y}_{0}', Y'_{1},
\ldots, Y'_{n}\bigr).
\end{eqnarray}
For each $k >0$, let $\succeq_{k}$ be the partial order on $[N]^k$ defined
by $x \succeq_{k} x'$ if and only if $x_i \geq x_i'$ for all $1 \leq i
\leq k$.
Since $\zeta$ is a monotone joining, ${\zeta}_{\mathrm{alt}}$
is a monotone joining, and we have that $X_i \succeq_{k_{\mathrm
{mark}}} Y_i$ for all $ 1 \leq i \leq n$,
and we also have that $\bar{X_0} \succeq_{k_{\mathrm{mark}}n_{\mathrm
{initial}}} \bar{Y_0}$ since
${\beta}_{\mathrm{sub}}$ is subordinate to $\Gamma_{\zeta}^{n_{\mathrm
{initial}}}$.
By the definition of the iterative star-coupling we also have that
\begin{eqnarray}
(\bar{X_0}, \bar{Y_0}) &\stackrel{d}
{=}& \bigl(\bar{X_0'}, \bar{Y_0'}
\bigr) \quad  \mbox{and}
\nonumber
\\[-8pt]
\label{d-block-dist} \\[-8pt]
\nonumber
({X_i}, Y_i) &\stackrel{d} {=}& \bigl(X_i',Y_i'
\bigr) \qquad \mbox{for all } 1 \leq i \leq n;
\end{eqnarray}
in particular, this implies that
\begin{equation}
\label{d-block-mono} \bar{X_0'} \succeq_{k_{\mathrm{mark}}n_{\mathrm{initial}}}
\bar{Y_0'}\quad \mbox{and}\quad X_i'
\succeq_{k_{\mathrm{mark}}} Y_i'.
\end{equation}

On each action block, by using the star-filler, we re-sample all of
its free intervals, using independent randomization on each action block.
Call ${\zeta}_{\mathrm{alt}*}$ the law of the resulting random variable
$\mathrm{Z}'$, the
\textit{alternating star-joining}.

%
\begin{lemma}
\label{weak-star-del}
Let $\zeta$ be an ergodic monotone joining of two Bernoulli measures
$\mu$ and $\nu$. The alternating star-joining ${\zeta}_{\mathrm{alt}*}$
is also an
ergodic monotone joining of $\mu$ and $\nu$. In addition, for any integer
$n_{\mathrm{rel}} \geq n_{\mathrm{initial}}$, let $R$ be the set of all
elements of $[N]^{\mathbb{Z}} \times[N]^{\mathbb{Z}}$ for which the
origin is contained
in an action block that contains
less than $n_{\mathrm{rel}}$ number of alternating intervals of size
$k_{\mathrm{mark}}$. Then
for any $\varepsilon>0$ for all sufficiently large
$k_{\mathrm{mark}}$ we have
\begin{equation}
\label{R} d^{*}(\zeta, {\zeta}_{\mathrm{alt}*}) < \varepsilon+
2{\zeta}_{\mathrm{alt}*}(R) + n_{\mathrm{initial}}/n_{\mathrm{rel}},
\end{equation}
where the inequality holds independently of the choice of $n_{\mathrm
{initial}}$ and the
probability measure ${\beta}_{\mathrm{sub}}$ that is subordinate
to $\Gamma^{n_{\mathrm{initial}}}$ and has the same marginals.
\end{lemma}

\begin{pf}
It follows from Proposition~\ref{alt-ball}, the definition of a star-filler,
\eqref{short-remark}, and~\eqref{d-block-mono} that ${\zeta}_{\mathrm
{alt}*}$ is an
ergodic monotone joining of $\mu$ and $\nu$.

The proof of inequality \eqref{R} is similar to that of Lemma~\ref
{weak-star-ball},
except for the following modification.
Note that ${\zeta}_{\mathrm{alt}*}(R) = {\zeta}_{\mathrm{alt}}(R)$ and
by \eqref{d-block-dist} that restricted to
every alternating interval of size $k_{\mathrm{mark}}$ that is not part
of an initial block,
${\zeta}_{\mathrm{alt}*}$ and ${\zeta}_{\mathrm{alt}}$ are equal; the
probability that an alternating interval
of size $k_{\mathrm{mark}}$ is part of an initial block is bounded by
$n_{\mathrm{initial}}/ n_{\mathrm{rel}}$,
when there are more than $n_{\mathrm{rel}}$ alternating
intervals of size $k_{\mathrm{mark}}$.
\end{pf}

We will show using Proposition~\ref{delJunco} that with a proper choice of
parameters that ${\zeta}_{\mathrm{alt}*}$ will be a suitable almost
factor and weak-star
close to $\zeta$.

\subsection{Baire category and the choice of parameters}

%
\begin{lemma}[(The Baire space)]
\label{baire-lemma}
Let $\mu=p^{\mathbb{Z}}$ and $\nu=q^{\mathbb{Z}}$ be two Bernoulli
measures on
$[N]^{\mathbb{Z}}$, where $p \succeq q$.
The space $\mathcal{M}= \mathcal{M}(\mu, \nu)$ of all monotone ergodic
joinings of $\mu$ and $\nu$ is a Baire space.
\end{lemma}
\begin{pf}
It is well known that space of all joinings of $\mu$ and $\nu$
is nonempty (since it contains the product measure), compact
and convex; furthermore its extreme points are the ergodic
joinings which form a (relatively) $G_{\delta}$ subset in
the space of all joinings of $\mu$ and $\nu$
\cite{down}, page 122, \cite{Rue}, Proposition~1.5.

Note that the subset of \emph{monotone}
joinings of $\mu$ and $\nu$ is closed and nonempty;
the ergodic monotone joining $\varrho^{\mathbb{Z}}$ is a witness to the
latter fact,
where $\varrho$ is the (monotone) quantile coupling of $p$ and $q$.
Hence $\mathcal M$ is a nonempty $G_{\delta}$ subset.

A $G_\delta$ subset of a complete metric space is a Polish space by a theorem
of Alexandrov \cite{Srivastava}, Theorem~2.2.1; and the
Baire category theorem tells us that
every Polish space is a Baire space \cite{Srivastava}, Theorem~2.5.5.
\end{pf}

Let $\mathcal F$ denote the product sigma-algebra for $[N]^{\mathbb{Z}}$.
Let
\[
\mathcal P := \{P_i\dvtx  0 \leq i \leq N-1\}
\]
denote the partition of $[N]^{\mathbb{Z}}$ according the zeroth coordinate
so that $P_i := \{x \in[N]^{\mathbb{Z}}\dvtx  x_0 = i\}$. Let $\zeta$ be a
joining of
the Bernoulli measures $\mu$ and $\nu$,
and let $\varepsilon> 0$. If for every set $P \in\sigma(\mathcal P)$
there exists a
$P' \in\mathcal F$ such that
\[
\zeta\bigl( \bigl(P' \times[N]^{\mathbb{Z}}\bigr)  \triangle
 \bigl([N]^{\mathbb{Z}} \times P\bigr) \bigr) < \varepsilon,
\]
then we say that $\zeta$ is an
\textit{$\varepsilon$-almost factor}.
%

%
\begin{lemma}
\label{factor}
Let $\zeta$ be a joining of two Bernoulli measures $\mu$ and $\nu$.
If $\zeta$ is an $\varepsilon$-almost factor for all $\varepsilon>0$,
then there exists a factor $\phi$ such that
\[
\zeta(F \times G) = \mu\bigl(F \cap\phi^{-1}(G)\bigr)
\qquad\mbox{for all $(F, G) \in\mathcal F \times\mathcal F$}.
\]
\end{lemma}

\begin{pf}
See \cite{Rue}, Theorem~2.8.
\end{pf}

Thus if $\zeta$ is an $\varepsilon$-almost factor for all $\varepsilon
>0$, then we say
that $\zeta$ is a \textit{factor}.

%
\begin{proposition}
\label{baire-set}
Let $\mu=p^{\mathbb{Z}}$ and $\nu=q^{\mathbb{Z}}$ be two Bernoulli
measures on
$[N]^{\mathbb{Z}}$, where $H(p) > H(q)$ and $p \succeq q$.
For each $n>0$, let $\mathcal E^n$ be the set of elements $\xi$ of
$\mathcal{M}$
that are $1/n$-almost factors. For each $n >0$, the following hold:
%
\begin{longlist}[(A)]
\item[(A)]
The set $\mathcal E^n$ is a relatively open subset of $\mathcal M$.
\item[(B)]
The set $\mathcal E^n$ is a dense subset of $\mathcal M$.
\end{longlist}
\end{proposition}

Proposition~\ref{baire-set}(A) is a standard argument
\cite{down}, pages 123--124,  that we give for completeness. The proof of
Proposition~\ref{baire-set}(B) will require the use of the
alternating star-joining.

\begin{pf*}{Proof of Theorem~\ref{result}}
By Proposition~\ref{baire-set}, the set defined by
\[
\mathcal E :=\bigcap_{n\geq1} \mathcal E^n
\]
is an intersection of relatively open dense subsets of $\mathcal M$.
By Lemma~\ref{baire-lemma}, $\mathcal E$ is a~nonempty subset of
$\mathcal M$,
and by Lemma~\ref{factor}, its elements are factors.
\end{pf*}

\begin{pf*}{Proof of Proposition~\ref{baire-set}(A)}
Let $\zeta\in\mathcal{E}^n$. Recall that $S$ is the left-shift.
Since the sigma-field $\mathcal F$ is generated by $\bigvee_{i \in
\mathbb{Z}}
S^{-i} \mathcal P$, and there are only a finite number of elements
in $\mathcal P$, there exists $m\in\mathbb{N}$ so that for all $P \in
\mathcal{P}$,
there is a~corresponding $P' \in \bigvee_{|i| <m} S^{-i} \mathcal P$
for which
%
\begin{equation}
\label{clopen} \zeta\bigl( \bigl(P' \times[N]^{\mathbb{Z}}
\bigr)  \triangle \bigl([N]^{\mathbb{Z}} \times P\bigr) \bigr) < 1/n.
\end{equation}
Note that \eqref{clopen} persists for all sufficiently small
perturbations of $\zeta$ since each $P$ and corresponding $P'$ are clopen sets.
\end{pf*}

\begin{pf*}{Proof of Proposition~\ref{baire-set}(B)}
Let $\zeta\in\mathcal M$, and $\varepsilon>0$ and $n >0$. We show
that with a
proper choice of parameters that for the alternating star-joining, we have
${\zeta}_{\mathrm{alt}*} \in\mathcal{E}^n$ and $d^{*}(\zeta, {\zeta
}_{\mathrm{alt}*}) < \varepsilon$.
Note that by Lemma~\ref{weak-star-del}, ${\zeta}_{\mathrm{alt}*}$ is a
monotone joining of
$\mu$ and $\nu$.
The following is a list of the parameters, chosen in order:
\begin{longlist}[(iii)]
\item
By reducing $\varepsilon$ if necessary, we may assume $\varepsilon< 1/n$.
\item
Set $\varepsilon' = \varepsilon/10$.
\item
Using Lemmas \ref{frozen}, \ref{mono-entropy} and \ref{weak-star-del}, choose
$k_{\mathrm{mark}}$ large enough so that the probability that the
origin is in a frozen
interval is less than $\varepsilon'$,
$h({\mu}_{\mathrm{fill}}) > h({\nu}_{\mathrm{fill}})$, and so for any
choice of $n_{\mathrm{initial}}$
and probability measure ${\beta}_{\mathrm{sub}}$ that is subordinate
to $\Gamma^{n_{\mathrm{initial}}}$ and has the same marginals, we have
\[
d^{*}(\zeta, {\zeta}_{\mathrm{alt}*}) < \varepsilon' +
2{\zeta}_{\mathrm{alt}*}(R) + n_{\mathrm{initial}}/n_{\mathrm{rel}},
\]
where $n_{\mathrm{rel}} \geq n_{\mathrm{initial}}$ and the set $R$ was
defined in Lemma~\ref{weak-star-del}.
\item
Appealing to Proposition~\ref{delJunco},
with $k_{\mathrm{mark}}$ as chosen above,
and $\eta= \varepsilon'$, we get an $n_{\mathrm{initial}}$ and a ${\beta
}_{\mathrm{sub}}$ which realizes
the conclusion of the proposition and in particular is subordinate to
$\Gamma_{\zeta}^{n_{\mathrm{initial}}}$ and has the same marginals.

%
\item
Choose an integer $n_{\mathrm{rel}} > n_{\mathrm{initial}}$ so that
$n_{\mathrm{initial}}/n_{\mathrm{rel}} < \varepsilon'$.
\item
Choose $r_{\mathrm{mark}}$ so that ${\zeta}_{\mathrm{alt}*}(R) <
\varepsilon'$ and so that the probability that
origin is not in an action block is less $\varepsilon'$.
\end{longlist}

By (i) it remains to argue that ${\zeta}_{\mathrm{alt}*}$ is an
$\varepsilon$-factor.
It suffices to define a deterministic function
$\psi\dvtx  [N]^{\mathbb{Z}} \to[N]$ so that
\[
{\zeta}_{\mathrm{alt}*} \bigl( (\mathrm{x}, \mathrm{y})\dvtx  (\mathrm{x}, \mathrm
{y}_0) = \bigl(\mathrm{x}, \psi(\mathrm{x})\bigr) \bigr) > 1 - \varepsilon.
\]
By the definition of ${\zeta}_{\mathrm{alt}*}$ and Proposition~\ref
{delJunco} and
(iv), it follows that $\psi$ can be easily defined from the
deterministic function $\Psi$ of the proposition, provided that
the origin is in a free interval on an action block. On the other hand,
by (iii) with probability less than $\varepsilon'$ the origin belongs
to a frozen interval, and by (vi) with probability greater than
$1-\varepsilon'$ the origin belongs to an action block where the
star-filling is
applied to the free intervals.
\end{pf*}

Finally, we discuss the proof of Theorem~\ref{relation}. Let $p$ and
$q$ be probability measures on $[N]$, and let $R$ be a relation on
$[N]$. Call a joining $\zeta$ of $\mu=p^{\mathbb{Z}}$ and
$\nu=q^{\mathbb{Z}}$ an \textit{$R$-joining} if
\[
\zeta\bigl\{ (\mathrm{x},\mathrm{y}) \in[N]^{\mathbb{Z}} \times
[N]^{\mathbb{Z}} \dvtx (\mathrm{x}_0, \mathrm{y}_0) \in R
\bigr\}=1.
\]
The proof of Theorem~\ref{relation} is the same as the proof of
Theorem~\ref{result}, except we work with $R$-couplings and
$R$-joinings instead of their monotone counterparts.

\begin{pf*}{Proof of Theorem~\ref{relation}}
We check the crucial details. By assumption there exists a probability
measure $\rho$ that is an $R$-coupling of $p$ and $q$.
Thus the set of $R$-joinings is nonempty. Given an $R$-joining of $\mu$
and $\nu$, the alternating joining is defined as before,
except instead of using the quantile coupling on individual
coordinates, we use the one given to us by assumption, $\rho$;
clearly the resulting alternating joining is still an $R$-joining. The
alternating star-joining is defined as before.
To check that it is a $R$-joining, we use the same facts that were
used to check monotonicity. The main point is
that the measure given Proposition~\ref{sub} is subordinate to the
original one, and del Junco's star-coupling is a coupling.
It follows from \eqref{d-block-dist} and the observation that if
$\alpha$ is an $R$-coupling, then any measure subordinate to $\alpha$
must also be an $R$-coupling.
\end{pf*}

\section*{Acknowledgment}
We thank Ay{\c{s}}e
\c{S}ahin for helpful conversations.

%





\printaddresses

\begin{thebibliography}{47}

\bibitem{MR525312}
%
\begin{barticle}[mr]
\bauthor{\bsnm{Akcoglu},~\bfnm{M.~A.}\binits{M.~A.}},
\bauthor{\bparticle{del} \bsnm{Junco},~\bfnm{A.}\binits{A.}} \AND
\bauthor{\bsnm{Rahe},~\bfnm{M.}\binits{M.}}
(\byear{1979}).
\btitle{Finitary codes between {M}arkov processes}.
\bjournal{Z. Wahrsch. Verw. Gebiete}
\bvolume{47}
\bpages{305--314}.
\bid{doi={10.1007/BF00535166}, issn={0044-3719}, mr={0525312}}
\end{barticle}
%

\bptok{imsref}%
\endbibitem

\bibitem{MR2736350}
%
\begin{barticle}[mr]
\bauthor{\bsnm{Angel},~\bfnm{Omer}\binits{O.}},
\bauthor{\bsnm{Holroyd},~\bfnm{Alexander~E.}\binits{A.~E.}} \AND
\bauthor{\bsnm{Soo},~\bfnm{Terry}\binits{T.}}
(\byear{2011}).
\btitle{Deterministic thinning of finite {P}oisson processes}.
\bjournal{Proc. Amer. Math. Soc.}
\bvolume{139}
\bpages{707--720}.
\bid{doi={10.1090/S0002-9939-2010-10535-X}, issn={0002-9939}, mr={2736350}}
\end{barticle}
%

\bptok{imsref}%
\endbibitem

\bibitem{Ball}
%
\begin{barticle}[mr]
\bauthor{\bsnm{Ball},~\bfnm{Karen}\binits{K.}}
(\byear{2005}).
\btitle{Monotone factors of i.i.d. processes}.
\bjournal{Israel J. Math.}
\bvolume{150}
\bpages{205--227}.
\bid{doi={10.1007/BF02762380}, issn={0021-2172}, mr={2255808}}
\end{barticle}
%

\bptok{imsref}%
\endbibitem

\bibitem{MR2133893}
%
\begin{barticle}[mr]
\bauthor{\bsnm{Ball},~\bfnm{Karen}\binits{K.}}
(\byear{2005}).
\btitle{Poisson thinning by monotone factors}.
\bjournal{Electron. Commun. Probab.}
\bvolume{10}
\bpages{60--69 (electronic)}.
\bid{doi={10.1214/ECP.v10-1134}, issn={1083-589X}, mr={2133893}}
\end{barticle}
%

\bptok{imsref}%
\endbibitem

\bibitem{MR0143862}
%
\begin{barticle}[mr]
\bauthor{\bsnm{Blum},~\bfnm{J.~R.}\binits{J.~R.}} \AND
\bauthor{\bsnm{Hanson},~\bfnm{D.~L.}\binits{D.~L.}}
(\byear{1963}).
\btitle{On the isomorphism problem for {B}ernoulli schemes}.
\bjournal{Bull. Amer. Math. Soc.}
\bvolume{69}
\bpages{221--223}.
\bid{issn={0002-9904}, mr={0143862}}
\end{barticle}
%

\bptok{imsref}%
\endbibitem

\bibitem{MR0092710}
%
\begin{barticle}[mr]
\bauthor{\bsnm{Breiman},~\bfnm{Leo}\binits{L.}}
(\byear{1957}).
\btitle{The individual ergodic theorem of information theory}.
\bjournal{Ann. Math. Statist.}
\bvolume{28}
\bpages{809--811}.
\bid{issn={0003-4851}, mr={0092710}}
\end{barticle}
%

\bptok{imsref}%
\endbibitem

\bibitem{BurRot}
%
\begin{bmisc}[author]
\bauthor{\bsnm{Burton},~\bfnm{R.}\binits{R.}} \AND
\bauthor{\bsnm{Rothstein},~\bfnm{A.}\binits{A.}}
(\byear{1977}).
\bhowpublished{Isomorphism theorems in ergodic theory.
Technical report, Oregon State Univ., Corvallis, OR.}
\end{bmisc}
%

\bptok{imsref}%
\endbibitem

\bibitem{BurKeaSer}
%
\begin{barticle}[mr]
\bauthor{\bsnm{Burton},~\bfnm{R.~M.}\binits{R.~M.}},
\bauthor{\bsnm{Keane},~\bfnm{M.~S.}\binits{M.~S.}} \AND
\bauthor{\bsnm{Serafin},~\bfnm{Jacek}\binits{J.}}
(\byear{2000}).
\btitle{Residuality of dynamical morphisms}.
\bjournal{Colloq. Math.}
\bvolume{84/85}
\bpages{307--317}.
\bid{issn={0010-1354}, mr={1784199}}
\end{barticle}
%

\bptok{imsref}%
\endbibitem

\bibitem{MR832433}
%
\begin{bbook}[mr]
\bauthor{\bsnm{Cornfeld},~\bfnm{I.~P.}\binits{I.~P.}},
\bauthor{\bsnm{Fomin},~\bfnm{S.~V.}\binits{S.~V.}} \AND
\bauthor{\bsnm{Sina{\u\i}},~\bfnm{Ya.~G.}\binits{Ya.~G.}}
(\byear{1982}).
\btitle{Ergodic Theory}.
\bseries{Grundlehren der Mathematischen Wissenschaften [Fundamental
Principles of Mathematical Sciences]}
\bvolume{245}.
\bpublisher{Springer},
\blocation{New York}.
\bnote{Translated from the Russian by A. B. Sosinski{\u\i}}.
\bid{doi={10.1007/978-1-4615-6927-5}, mr={0832433}}
\end{bbook}
%

\bptok{imsref}%
\endbibitem

\bibitem{Juncoa}
%
\begin{barticle}[mr]
\bauthor{\bparticle{del} \bsnm{Junco},~\bfnm{Andr{\'e}s}\binits{A.}}
(\byear{1981}).
\btitle{Finitary codes between one-sided {B}ernoulli shifts}.
\bjournal{Ergodic Theory Dynam. Systems}
\bvolume{1}
\bpages{285--301}.
\bid{issn={0143-3857}, mr={0662471}}
\end{barticle}
%

\bptok{imsref}%
\endbibitem

\bibitem{Junco}
%
\begin{barticle}[mr]
\bauthor{\bparticle{del} \bsnm{Junco},~\bfnm{Andr{\'e}s}\binits{A.}}
(\byear{1990}).
\btitle{Bernoulli shifts of the same entropy are finitarily and
unilaterally isomorphic}.
\bjournal{Ergodic Theory Dynam. Systems}
\bvolume{10}
\bpages{687--715}.
\bid{doi={10.1017/S014338570000585X}, issn={0143-3857}, mr={1091422}}
\end{barticle}
%

\bptok{imsref}%
\endbibitem

\bibitem{MR571401}
%
\begin{barticle}[mr]
\bauthor{\bsnm{Denker},~\bfnm{Manfred}\binits{M.}} \AND
\bauthor{\bsnm{Keane},~\bfnm{Michael}\binits{M.}}
(\byear{1979}).
\btitle{Almost topological dynamical systems}.
\bjournal{Israel J. Math.}
\bvolume{34}
\bpages{139--160}.
\bid{doi={10.1007/BF02761830}, issn={0021-2172}, mr={0571401}}
\end{barticle}
%

\bptok{imsref}%
\endbibitem

\bibitem{Rue}
%
\begin{barticle}[mr]
\bauthor{\bsnm{de~la Rue},~\bfnm{Thierry}\binits{T.}}
(\byear{2006}).
\btitle{An introduction to joinings in ergodic theory}.
\bjournal{Discrete Contin. Dyn. Syst.}
\bvolume{15}
\bpages{121--142}.
\bid{doi={10.3934/dcds.2006.15.121}, issn={1078-0947}, mr={2191388}}
\end{barticle}
%

\bptok{imsref}%
\endbibitem

\bibitem{down}
%
\begin{bbook}[mr]
\bauthor{\bsnm{Downarowicz},~\bfnm{Tomasz}\binits{T.}}
(\byear{2011}).
\btitle{Entropy in Dynamical Systems}.
\bseries{New Mathematical Monographs}
\bvolume{18}.
\bpublisher{Cambridge Univ. Press},
\blocation{Cambridge}.
\bid{doi={10.1017/CBO9780511976155}, mr={2809170}}
\end{bbook}
%

\bptok{imsref}%
\endbibitem

\bibitem{Dudley}
%
\begin{bbook}[author]
\bauthor{\bsnm{Dudley},~\bfnm{R.}\binits{R.}}
(\byear{1989}).
\btitle{Real Analysis and Probability}.
\bpublisher{Wadsworth},
\blocation{Pacific Grove, CA}.
\end{bbook}
%

\bptok{imsref}%
\endbibitem

\bibitem{Evans}
%
\begin{bincollection}[mr]
\bauthor{\bsnm{Evans},~\bfnm{Steven~N.}\binits{S.~N.}}
(\byear{2010}).
\btitle{A zero--one law for linear transformations of L\'evy noise}.
In \bbooktitle{Algebraic Methods in Statistics and Probability {II}}.
\bseries{Contemp. Math.}
\bvolume{516}
\bpages{189--197}.
\bpublisher{Amer. Math. Soc.},
\blocation{Providence, RI}.
\bid{doi={10.1090/conm/516/10175}, mr={2730749}}
\end{bincollection}
%

\bptok{imsref}%
\endbibitem

\bibitem{GGP}
%
\begin{barticle}[mr]
\bauthor{\bsnm{Gurel-Gurevich},~\bfnm{Ori}\binits{O.}} \AND
\bauthor{\bsnm{Peled},~\bfnm{Ron}\binits{R.}}
(\byear{2013}).
\btitle{Poisson thickening}.
\bjournal{Israel J. Math.}
\bvolume{196}
\bpages{215--234}.
\bid{doi={10.1007/s11856-012-0181-2}, issn={0021-2172}, mr={3096589}}
\end{barticle}
%

\bptok{imsref}%
\endbibitem

\bibitem{Hall01011935}
%
\begin{barticle}[author]
\bauthor{\bsnm{Hall},~\bfnm{P.}\binits{P.}}
(\byear{1935}).
\btitle{On representatives of subsets}.
\bjournal{J. London Math. Soc.(1)}
\bvolume{10}
\bpages{26--30}.
\end{barticle}
%

\bptok{imsref}%
\endbibitem

\bibitem{MR0122959}
%
\begin{barticle}[mr]
\bauthor{\bsnm{Halmos},~\bfnm{Paul~R.}\binits{P.~R.}}
(\byear{1961}).
\btitle{Recent progress in ergodic theory}.
\bjournal{Bull. Amer. Math. Soc.}
\bvolume{67}
\bpages{70--80}.
\bid{issn={0002-9904}, mr={0122959}}
\end{barticle}
%

\bptok{imsref}%
\endbibitem

\bibitem{MR2308235}
%
\begin{barticle}[mr]
\bauthor{\bsnm{Harvey},~\bfnm{Nate}\binits{N.}},
\bauthor{\bsnm{Holroyd},~\bfnm{Alexander~E.}\binits{A.~E.}},
\bauthor{\bsnm{Peres},~\bfnm{Yuval}\binits{Y.}} \AND
\bauthor{\bsnm{Romik},~\bfnm{Dan}\binits{D.}}
(\byear{2007}).
\btitle{Universal finitary codes with exponential tails}.
\bjournal{Proc. Lond. Math. Soc. (3)}
\bvolume{94}
\bpages{475--496}.
\bid{doi={10.1112/plms/pdl018}, issn={0024-6115}, mr={2308235}}
\end{barticle}
%

\bptok{imsref}%
\endbibitem

\bibitem{MR2884878}
%
\begin{barticle}[mr]
\bauthor{\bsnm{Holroyd},~\bfnm{Alexander~E.}\binits{A.~E.}},
\bauthor{\bsnm{Lyons},~\bfnm{Russell}\binits{R.}} \AND
\bauthor{\bsnm{Soo},~\bfnm{Terry}\binits{T.}}
(\byear{2011}).
\btitle{Poisson splitting by factors}.
\bjournal{Ann. Probab.}
\bvolume{39}
\bpages{1938--1982}.
\bid{doi={10.1214/11-AOP651}, issn={0091-1798}, mr={2884878}}
\end{barticle}
%

\bptok{imsref}%
\endbibitem

\bibitem{MR2342699}
%
\begin{barticle}[mr]
\bauthor{\bsnm{Katok},~\bfnm{Anatole}\binits{A.}}
(\byear{2007}).
\btitle{Fifty years of entropy in dynamics: 1958--2007}.
\bjournal{J. Mod. Dyn.}
\bvolume{1}
\bpages{545--596}.
\bid{doi={10.3934/jmd.2007.1.545}, issn={1930-5311}, mr={2342699}}
\end{barticle}
%

\bptok{imsref}%
\endbibitem

\bibitem{keanea}
%
\begin{barticle}[mr]
\bauthor{\bsnm{Keane},~\bfnm{M.}\binits{M.}} \AND
\bauthor{\bsnm{Smorodinsky},~\bfnm{M.}\binits{M.}}
(\byear{1977}).
\btitle{A class of finitary codes}.
\bjournal{Israel J. Math.}
\bvolume{26}
\bpages{352--371}.
\bid{issn={0021-2172}, mr={0450514}}
\end{barticle}\vadjust{\goodbreak}
%

\bptok{imsref}%
\endbibitem

\bibitem{keaneb}
%
\begin{barticle}[mr]
\bauthor{\bsnm{Keane},~\bfnm{Michael}\binits{M.}} \AND
\bauthor{\bsnm{Smorodinsky},~\bfnm{Meir}\binits{M.}}
(\byear{1979}).
\btitle{Bernoulli schemes of the same entropy are finitarily isomorphic}.
\bjournal{Ann. of Math. (2)}
\bvolume{109}
\bpages{397--406}.
\bid{doi={10.2307/1971117}, issn={0003-486X}, mr={0528969}}
\end{barticle}
%

\bptok{imsref}%
\endbibitem

\bibitem{Krieger1}
%
\begin{barticle}[mr]
\bauthor{\bsnm{Krieger},~\bfnm{Wolfgang}\binits{W.}}
(\byear{1970}).
\btitle{On entropy and generators of measure-preserving transformations}.
\bjournal{Trans. Amer. Math. Soc.}
\bvolume{149}
\bpages{453--464}.
\bid{issn={0002-9947}, mr={0259068}}
\end{barticle}
%

\bptok{imsref}%
\endbibitem

\bibitem{MR0110782}
%
\begin{barticle}[mr]
\bauthor{\bsnm{Me{\v{s}}alkin},~\bfnm{L.~D.}\binits{L.~D.}}
(\byear{1959}).
\btitle{A case of isomorphism of {B}ernoulli schemes}.
\bjournal{Dokl. Akad. Nauk SSSR}
\bvolume{128}
\bpages{41--44}.
\bid{issn={0002-3264}, mr={0110782}}
\end{barticle}
%

\bptok{imsref}%
\endbibitem

\bibitem{mester}
%
\begin{barticle}[mr]
\bauthor{\bsnm{Mester},~\bfnm{P{\'e}ter}\binits{P.}}
(\byear{2013}).
\btitle{Invariant monotone coupling need not exist}.
\bjournal{Ann. Probab.}
\bvolume{41}
\bpages{1180--1190}.
\bid{doi={10.1214/12-AOP767}, issn={0091-1798}, mr={3098675}}
\end{barticle}
%

\bptok{imsref}%
\endbibitem

\bibitem{tom}
%
\begin{barticle}[mr]
\bauthor{\bsnm{Meyerovitch},~\bfnm{Tom}\binits{T.}}
(\byear{2013}).
\btitle{Ergodicity of {P}oisson products and applications}.
\bjournal{Ann. Probab.}
\bvolume{41}
\bpages{3181--3200}.
\bid{doi={10.1214/12-AOP824}, issn={0091-1798}, mr={3127879}}
\end{barticle}
%

\bptok{imsref}%
\endbibitem

\bibitem{Ornstein}
%
\begin{barticle}[mr]
\bauthor{\bsnm{Ornstein},~\bfnm{Donald}\binits{D.}}
(\byear{1970}).
\btitle{Bernoulli shifts with the same entropy are isomorphic}.
\bjournal{Adv. Math.}
\bvolume{4}
\bpages{337--352}.
\bid{issn={0001-8708}, mr={0257322}}
\end{barticle}
%

\bptok{imsref}%
\endbibitem

\bibitem{MR3052869}
%
\begin{barticle}[mr]
\bauthor{\bsnm{Ornstein},~\bfnm{Don}\binits{D.}}
(\byear{2013}).
\btitle{Newton's laws and coin tossing}.
\bjournal{Notices Amer. Math. Soc.}
\bvolume{60}
\bpages{450--459}.
\bid{doi={10.1090/noti974}, issn={0002-9920}, mr={3052869}}
\end{barticle}
%

\bptok{imsref}%
\endbibitem

\bibitem{MR0447525}
%
\begin{bbook}[mr]
\bauthor{\bsnm{Ornstein},~\bfnm{Donald~S.}\binits{D.~S.}}
(\byear{1974}).
\btitle{Ergodic Theory, Randomness, and Dynamical Systems}.
\bseries{Yale Mathematical Monographs}
\bvolume{5}.
\bpublisher{Yale Univ. Press, New Haven},
\blocation{CT}.
\bid{mr={0447525}}
\end{bbook}
%

\bptok{imsref}%
\endbibitem

\bibitem{MR0412386}
%
\begin{barticle}[mr]
\bauthor{\bsnm{Ornstein},~\bfnm{Donald~S.}\binits{D.~S.}} \AND
\bauthor{\bsnm{Weiss},~\bfnm{Benjamin}\binits{B.}}
(\byear{1975}).
\btitle{Unilateral codings of {B}ernoulli systems}.
\bjournal{Israel J. Math.}
\bvolume{21}
\bpages{159--166}.
\bid{issn={0021-2172}, mr={0412386}}
\end{barticle}
%

\bptok{imsref}%
\endbibitem

\bibitem{MR910005}
%
\begin{barticle}[mr]
\bauthor{\bsnm{Ornstein},~\bfnm{Donald~S.}\binits{D.~S.}} \AND
\bauthor{\bsnm{Weiss},~\bfnm{Benjamin}\binits{B.}}
(\byear{1987}).
\btitle{Entropy and isomorphism theorems for actions of amenable groups}.
\bjournal{J. Anal. Math.}
\bvolume{48}
\bpages{1--141}.
\bid{doi={10.1007/BF02790325}, issn={0021-7670}, mr={0910005}}
\end{barticle}
%

\bptok{imsref}%
\endbibitem

\bibitem{MR1073173}
%
\begin{bbook}[mr]
\bauthor{\bsnm{Petersen},~\bfnm{Karl}\binits{K.}}
(\byear{1989}).
\btitle{Ergodic Theory}.
\bseries{Cambridge Studies in Advanced Mathematics}
\bvolume{2}.
\bpublisher{Cambridge Univ. Press},
\blocation{Cambridge}.
\bnote{Corrected reprint of the 1983 original}.
\bid{mr={1073173}}
\end{bbook}
%

\bptok{imsref}%
\endbibitem

\bibitem{MR1164596}
%
\begin{barticle}[mr]
\bauthor{\bsnm{Propp},~\bfnm{James~Gary}\binits{J.~G.}}
(\byear{1991}).
\btitle{Coding {M}arkov chains from the past}.
\bjournal{Israel J. Math.}
\bvolume{75}
\bpages{289--328}.
\bid{doi={10.1007/BF02776030}, issn={0021-2172}, mr={1164596}}
\end{barticle}
%

\bptok{imsref}%
\endbibitem

\bibitem{QSb}
%
\begin{barticle}[author]
\bauthor{\bsnm{Quas},~\bfnm{Anthony}\binits{A.}} \AND
\bauthor{\bsnm{Soo},~\bfnm{Terry}\binits{T.}}
(\byear{2014}).
\btitle{Ergodic universality of some topological dynamical systems}.
\bjournal{Trans. Amer. Math. Soc.}
\bnote{To appear. Available at \url{http://arxiv.org/abs/1208.3501}.}
\end{barticle}
%

\bptok{imsref}%
\endbibitem

\bibitem{MR633760}
%
\begin{bincollection}[mr]
\bauthor{\bsnm{Rudolph},~\bfnm{Daniel~J.}\binits{D.~J.}}
(\byear{1981}).
\btitle{A characterization of those processes finitarily isomorphic to
a {B}ernoulli shift}.
In \bbooktitle{Ergodic Theory and Dynamical Systems, I ({C}ollege
{P}ark, {M}d., 1979--1980)}.
\bseries{Progr. Math.}
\bvolume{10}
\bpages{1--64}.
\bpublisher{Birkh\"auser},
\blocation{Boston, MA.}
\bid{mr={0633760}}
\end{bincollection}
%

\bptok{imsref}%
\endbibitem

\bibitem{Rud}
%
\begin{bbook}[mr]
\bauthor{\bsnm{Rudolph},~\bfnm{Daniel~J.}\binits{D.~J.}}
(\byear{1990}).
\btitle{Fundamentals of Measurable Dynamics: Ergodic Theory on Lebesgue Spaces}.
\bpublisher{Clarendon},
\blocation{New York}.
\bid{mr={1086631}}
\end{bbook}
%

\bptok{imsref}%
\endbibitem

\bibitem{MR2306207}
%
\begin{bincollection}[mr]
\bauthor{\bsnm{Serafin},~\bfnm{Jacek}\binits{J.}}
(\byear{2006}).
\btitle{Finitary codes, a short survey}.
In \bbooktitle{Dynamics \& Stochastics}.
\bseries{Institute of Mathematical Statistics Lecture Notes---Monograph Series}
\bvolume{48}
\bpages{262--273}.
\bpublisher{IMS},
\blocation{Beachwood, OH}.
\bid{doi={10.1214/lnms/1196285827}, mr={2306207}}
\end{bincollection}
%

\bptok{imsref}%
\endbibitem

\bibitem{shea}
%
\begin{barticle}[mr]
\bauthor{\bsnm{Shea},~\bfnm{Stephen}\binits{S.}}
(\byear{2013}).
\btitle{Finitarily {B}ernoulli factors are dense}.
\bjournal{Fund. Math.}
\bvolume{223}
\bpages{49--54}.
\bid{doi={10.4064/fm223-1-3}, issn={0016-2736}, mr={3125131}}
\end{barticle}
%

\bptok{imsref}%
\endbibitem

\bibitem{MR2876281}
%
\begin{barticle}[mr]
\bauthor{\bsnm{Shea},~\bfnm{Stephen~M.}\binits{S.~M.}}
(\byear{2012}).
\btitle{On the marker method for constructing finitary isomorphisms}.
\bjournal{Rocky Mountain J. Math.}
\bvolume{42}
\bpages{293--304}.
\bid{doi={10.1216/RMJ-2012-42-1-293}, issn={0035-7596}, mr={2876281}}
\end{barticle}
%

\bptok{imsref}%
\endbibitem

\bibitem{Sinai}
%
\begin{barticle}[mr]
\bauthor{\bsnm{Sina{\u\i}},~\bfnm{Ja.~G.}\binits{Ja.~G.}}
(\byear{1964}).
\btitle{On a weak isomorphism of transformations with invariant measure}.
\bjournal{Mat. Sb. (N.S.)}
\bvolume{63 (105)}
\bpages{23--42}.
\bid{mr={0161961}}
\end{barticle}
%

\bptok{imsref}%
\endbibitem

\bibitem{MR2766434}
%
\begin{bbook}[mr]
\bauthor{\bsnm{Sinai},~\bfnm{Yakov~G.}\binits{Y.~G.}}
(\byear{2010}).
\btitle{Selecta. {V}olume I. {E}rgodic Theory and Dynamical Systems}.
\bpublisher{Springer},
\blocation{New York}.
\bid{doi={10.1007/978-0-387-87870-6}, mr={2766434}}
\end{bbook}
%

\bptok{imsref}%
\endbibitem

\bibitem{Srivastava}
%
\begin{bbook}[mr]
\bauthor{\bsnm{Srivastava},~\bfnm{S.~M.}\binits{S.~M.}}
(\byear{1998}).
\btitle{A Course on {B}orel Sets}.
\bseries{Graduate Texts in Mathematics}
\bvolume{180}.
\bpublisher{Springer},
\blocation{New York}.
\bid{doi={10.1007/978-3-642-85473-6}, mr={1619545}}
\end{bbook}
%

\bptok{imsref}%
\endbibitem

\bibitem{Strassen}
%
\begin{barticle}[mr]
\bauthor{\bsnm{Strassen},~\bfnm{V.}\binits{V.}}
(\byear{1965}).
\btitle{The existence of probability measures with given marginals}.
\bjournal{Ann. Math. Statist.}
\bvolume{36}
\bpages{423--439}.
\bid{issn={0003-4851}, mr={0177430}}
\end{barticle}
%

\bptok{imsref}%
\endbibitem

\bibitem{MR1741181}
%
\begin{bbook}[mr]
\bauthor{\bsnm{Thorisson},~\bfnm{Hermann}\binits{H.}}
(\byear{2000}).
\btitle{Coupling, Stationarity, and Regeneration}.
\bpublisher{Springer},
\blocation{New York}.
\bid{doi={10.1007/978-1-4612-1236-2}, mr={1741181}}
\end{bbook}
%

\bptok{imsref}%
\endbibitem
\

\bibitem{MR0304616}
%
\begin{barticle}[mr]
\bauthor{\bsnm{Weiss},~\bfnm{Benjamin}\binits{B.}}
(\byear{1972}).
\btitle{The isomorphism problem in ergodic theory}.
\bjournal{Bull. Amer. Math. Soc.}
\bvolume{78}
\bpages{668--684}.
\bid{issn={0002-9904}, mr={0304616}}
\end{barticle}
%

\bptok{imsref}%
\endbibitem
\end{thebibliography}
\end{document}